\pdfoutput=1
\documentclass[11pt]{amsart}
\usepackage[utf8]{inputenc}
\usepackage[english]{babel}
\usepackage{cite}
\usepackage{amsthm,amsmath,amssymb,hyperref,amsfonts,tikz,adjustbox}
\usetikzlibrary{cd}
\numberwithin{equation}{section}
\usepackage[shortlabels]{enumitem}
\usepackage[a4paper, total={6in, 8in}]{geometry}
\usepackage{imakeidx}

\newtheorem{theorem}{Theorem}[section] 
\newtheorem{example}[theorem]{Example}
\newtheorem{corollary}[theorem]{Corollary} 
\newtheorem{lemma}[theorem]{Lemma}
\newtheorem{proposition}[theorem]{Proposition}

\theoremstyle{definition} \newtheorem{definition}[theorem]{Definition}

\theoremstyle{remark} \newtheorem{remark}[theorem]{Remark}

\newcommand{\m}{{}^{-1}} \newcommand{\G}{\Gamma} 
\newcommand{\s}{\sigma}
 \newcommand{\0}{\theta} \newcommand{\de}{\delta}
 \newcommand{\af}{\alpha}

 \newcommand{\U}{{\mathcal U}}

\newcommand{\Kpar}{{K}_{\rm par}}

\newcommand{\kpar}{{\kappa }_{\rm par}}

\title[The twisted partial group algebra and (co)homology]{The twisted partial group algebra and (co)homology  of partial crossed products}

  % 18G40 corresponds to spectral sequences in homological algebra.

\author{Mikhailo Dokuchaev}
\address[Mikhailo Dokuchaev]{Departamento de Matem{\'a}tica, Universidade de S\~ao Paulo,
Rua do Mat\~ao, 1010, 05508-090 S\~ao Paulo, Brazil}
\email{dokucha@ime.usp.br}  

\author{Emmanuel Jerez}
\address[Emmanuel Jerez]{Departamento de Matem{\'a}tica, Universidade de S\~ao Paulo,
Rua do Mat\~ao, 1010, 05508-090 S\~ao Paulo, Brazil}
\email{ejerez@ime.usp.br}

\begin{document}

 \begin{abstract} Given a group $G$ and a partial factor set $\sigma $ of $G,$ we introduce the twisted partial group algebra $\kappa_{par}^{\sigma}G,$ which governs the partial projective $\s $-representations of $G$ into algebras over a filed $\kappa .$ Using the relation between partial projective representations and twisted partial actions we endow $\kpar^\sigma G$ with the structure of a crossed product by a twisted partial action of $G$ on a commutative subalgebra of $\kpar^\sigma G.$ Then, we use twisted partial group algebras to obtain a first quadrant Grothendieck spectral sequence converging to the Hochschild ho\-mo\-lo\-gy of the crossed product $A\ast_{\Theta} G,$ involving the Hochschild homology of $A$ and the partial homology of $G,$ where ${\Theta}$ is a unital twisted partial action of $G$ on a $\kappa$-algebra $A$ with a $\kappa $-based twist. An analogous third quadrant cohomological spectral sequence is also obtained.
\end{abstract}

\maketitle

 \section{Introduction}

 Partial group algebras govern partial group representations. Both concepts were initially introduced in the theory of $C^*$-algebras
 \cite{exe}, \cite{QR}, which together with the related notion of a partial group action \cite{exe} stimulated various algebraic developments. Partial actions are abundant in Mathematics. 
 In particular, a flow of a differentiable vector field on a manifold is a partial action of the additive group of the real numbers \cite{AbadieTwo}. The latter is a classical example of a local action of a Lie group in Differential Geometry, and the notion of a maximum local transformation group given by 
R. S. Palais in \cite[p. 65]{Palais} is, in fact, the concept of a partial group action formulated in the differentiable setting.

 $C^*$-algebraic early applications of partial actions and partial representations include those to algebras generated by partial isometries, such as the Cuntz-Krieger $C^*$-algebras \cite{E2}, \cite{QR}, \cite{ELQ}, the Exel-Laca $C^*$-algebras \cite{EL}, as well as to graded $C^*$-algebras \cite{E0}. The first algebraic advances were made in 
 \cite{DEP}, \cite{KL1}, \cite{St1} and \cite{St2}.
In algebra, the new concepts are useful to graded algebras, Hecke algebras, Leavitt path algebras, Steinberg algebras,
 Thompson's groups, inverse semigroups, restriction semigroups and automata (see the survey article \cite{D3}). More recent applications were obtained to $C^*$-algebras related to dynamical systems of type $(m,n)$ \cite{AraEKa}, to dynamical systems associated to separated graphs and related $C^*$-algebras \cite{AraE1}, \cite{AraL}, to paradoxical decompositions \cite{AraE1}, to shifts \cite{AraL}, to full or reduced $C^*$-algebras of $E$-unitary or strongly $E^*$-unitary inverse semigroups \cite{MiSt}, 
to topological higher-rank graphs \cite{RenWil}, to Matsumoto and Carlsen-Matsumoto $C^*$-algebras of arbitrary subshifts \cite{DE2}, to ultragraph $C^*$-algebras \cite{GR3} and to expansions of monoids in the class of two-sided restriction monoids \cite{Kud2} (see also the related earlier results in \cite{CornGould} and \cite{Kud}). Inspired by partial actions a suitable generalization of them was defined in \cite{KudLaan} and applied to Ehresmann semigroups. The latter has its origin in Differential Geometry and was introduced in \cite{Lawson1991}.

Partial actions and, more generally, twisted partial actions on $C^*$-algebras were defined in order to introduce a wider notion of a crossed product and apply it to deal with important classes of $C^*$-algebras (see \cite{E6} and \cite{D3}). Besides the remarkable $C^*$-algebraic developments a good deal of interesting algebraic research was done to study crossed products by partial actions (partial skew group rings) considering their ring theoretic properties, ideal structure, and applications 
(see \cite{D3}). As to the more general crossed products by twisted partial actions, we mention their importance for graded algebras \cite{BaMaPi}, \cite{DES1}, \cite{NysOinPin2}, for Hecke algebras \cite{E3}, the study of their ring theoretical properties \cite{BLP}, \cite{BaravCoSo}, \cite{BeCo}, \cite{BeCoFeFlo}, \cite{LaLuOiWa}, \cite{PSantA}, and the relevance of their semigroup theoretical version, which is useful for extensions related to group cohomology based on partial actions \cite{DKh2}, \cite{DKh3}. The latter was introduced in \cite{DKh} and found applications to partial projective representations and partial Galois extensions of commutative rings (see \cite{D3}). Furthermore, partial $2$-cocycles 
appeared in \cite{KennedySchafhauser} as certain obstructions in the study of the ideal structure of the reduced crossed product of a $C^*$-algebra by a global action. In addition, the cohomology from \cite{DKh} was extended from groups to groupoids in \cite{NysOinPin3} and used to 
classify the equivalence classes of certain groupoid graded rings. It also inspired the treatment of partial cohomology from the point of view of Hopf algebras \cite{BaMoTe}.

Besides the group cohomology based on partial actions, a cohomology related to partial group representations was defined in 
\cite{AlAlRePartialCohomology}. Despite the similarity in formulas and the interactions between partial actions and partial representations, more differences than similarities can be seen between the two partial cohomologies so far (see \cite{DKhS2}). The partial group cohomology in the sense of \cite{AlAlRePartialCohomology} and a similarly defined partial group homology was recently studied in \cite{MMAMDDHKPartialHomology}.

Given a unital partial action $\0 $ of a group $G$ on an algebra $A$ over a field $\kappa $ (see Definition~\ref{def:parac}), a third quadrant Grothendieck
spectral sequence was produced in \cite{AlAlRePartialCohomology} relating the Hochschild co\-ho\-mo\-lo\-gy of the partial skew group algebra $A \rtimes _{\0}G$ with the Hochschild cohomology of $A$ and the partial cohomology of $G$ (in the sense of \cite{AlAlRePartialCohomology}). Our purpose in this paper is to
take as a starting point partial projective group representations,
which allows us to deal with the more general crossed product $A\ast _ {\Theta} G,$ where ${\Theta}$ is a unital twisted partial action of $G$ on $A$ (see Definition~\ref{def:twparac}), whose twist is 
$\kappa $-based (see Definition~\ref{def:par2-cocycle}), considering both homology and cohomology.

The theory of partial projective group representations was developed in \cite{DN}, \cite{DN2}, \cite{DoNoPi}, \cite{DoSa}, with further results in \cite{NovP},  \cite{Pi5}, \cite{Pi}. 
They were defined in \cite{DN} as certain maps of the form $G\to M$, where $M$ is a $\kappa $-cancellative monoid (see Definition~\ref{def:K-smgrp}) over a field $\kappa.$ In this article it is appropriate to restrict ourselves 
to the case, in which $M$ is a $\kappa $-algebra. It is well known that given a $\kappa $-valued 2-cocycle (equivalently, a factor set) $\tau $ of a group $G,$ the usual projective $\tau$-representations of $G$ are in bijective correspondence with the algebra homomorphisms of the twisted group algebra 
$\kappa ^{\tau} G.$ Likewise, we introduce an algebra
 $\kappa_{par}^{\sigma}G$, called the twisted partial group algebra of $G$, which governs the partial projective $\sigma $-representations of $G$ for a fixed partial factor set $\sigma .$
 It generalizes the notion of the partial group algebra $\kappa_{par}G$ of $G$ (see \cite{DEP}), which controls the partial representations of $G.$

In Section~\ref{sec:ParProjRepr} we establish some notions and recall some facts on partial projective representations, and point out that we do not lose factor sets of partial projective representations when restricting ourselves to the case, in which they take values in algebras (see Remark~\ref{rem:ParFactorSets}).
The concept of the twisted partial group algebra 
$\kappa_{par}^{\sigma}G$ of $G$ for a given partial factor set $\s $ is introduced in Definition~\ref{def:ParTwiGrAlg}. In Proposition~\ref{prop:1to1corresp} we establish the universal property of $\kappa_{par}^{\sigma}G$ with respect to the partial projective $\s$-representations of $G$ in $\kappa $-algebras. For further use, we also define an involution $\rho \mapsto \rho ^*$ on the semigroup 
 $pm(G)$ of all factor sets of partial projective representations of $G$ in $\kappa$-algebras, such that
the algebras $\kappa_{par}^{\rho^*}G$ and $(\kappa_{par}^{\rho}G)^{op}$ are isomorphic, where $(\kappa_{par}^{\rho}G)^{op}$ stands for the opposite algebra of $\kappa_{par}^{\rho}G$ (see Proposition~\ref{p_left_and_right_partial_modules_equivalence}).

Section~\ref{sec:PPRvsTwistedPA} is devoted to the relation between partial projective representations and twisted partial actions. In particular, Corollary~\ref{cor:FactorSets=Twists} states that the set of all partial factor sets $pm(G)$ of partial projective representations of $G$ in $\kappa$-algebras coincides with the set of all $\kappa $-valued twists of unital partial actions of $G$ on $\kappa$-algebras. For a given $\sigma \in pm(G)$ we define the notion of a covariant $\sigma $-representation 
of a unital partial action, which is a tool to produce homomorphisms from the crossed product $ A\ast _{\0, \s} G$ into algebras, where
$\0$ is a unital partial action of $G$ on $A,$ for which $\sigma $ is a twist (see Proposition~\ref{prop:covariant}). This is a ``$\sigma$-version'' of a known tool efficiently used for partial skew group algebras \cite{E6}. We apply the obtained results in order to endow in Theorem~\ref{prop:kparCrossed} the partial skew group algebra $\kpar ^\sigma G$ with a crossed product structure of the form $B^\sigma \ast _{(\0^\sigma, \s ) } G,$ where $\0 ^\sigma $ is a unital partial action of $G$ on a commutative subalgebra $B^\sigma $ of $\kpar ^\sigma G,$ for which $\s $ is a twist. 
%The subalgebra 
% $B^\sigma $ is generated by idempotents defined
% using the canonical generators of $\kpar ^\sigma G$ (see Definition~\ref{def:ParTwiGrAlg}).
 
 After giving some useful facts in Section~\ref{sec:Remarks}
 on bimodules related to partial projective representations and on equivalence of partial factor sets, we establish a homological spectral sequence in Section~\ref{sec:HomolSpecSeq}, which converges the Hochschild homology of the crossed product
 $A \ast _{\0, \s } G,$ where $\0$ is a unital partial action of $G$ on a $\kappa $-algebra $A$ with twist $\s \in pm(G).$ It is done by establishing a series of preparatory facts. The functors of the form $ F_1(-)= A \otimes_{A^e}- $ and $ F(-)=A \ast _{\0, \s } G \otimes_{(A \ast _{\0, \s } G)^e}-$ come into the picture naturally, as their left derived functors compute the Hochschild homology of $A$ and 
 $A \ast _{\0, \s } G,$ respectively (here $A^e$ denotes the enveloping algebra $A\otimes A^{op}$ of $A$ and the meaning of $(A \ast _{\0, \s } G)^e$ is similar). We also consider the functor 
$ F_2(-)= B^{\sigma} \otimes_{\kappa_{par}^{\sigma''}G}-$ applied to the left $\kappa_{par}^{\sigma''}G$-modules, where 
 $\sigma ''$ is an idempotent partial factor set equivalent to $ \s \s ^*.$ One of the main steps is to prove that the functors $F_2F_1$ and $F$ are naturally isomorphic when applied to the ca\-te\-go\-ry of $A \ast _{\0, \s } G$-bimodules (see Corollary~\ref{c_F2F1F}). In Theorem~\ref{t_SSHHTor} we obtain a spectral sequence involving the functor $\operatorname{Tor}^{\kappa_{par}^{\sigma''}G}_p (B^\sigma , - )$ and converging to
 the Hochschild homology of
 $A \ast _{\0, \s } G.$ Next we relate this $\operatorname{Tor}$ functor with the partial homology of $G$ in Proposition~\ref{prop:TorSigma'Hkpar}, and the main result of the section is 
 Theorem~\ref{t_SSHH}, which for any $A \ast_{\theta,\sigma} G$-bimodule $M$ establishes 
 a first quadrant homological spectral sequence of the form
 \[
    E^2_{p,q} = H_p^{par}(G,H_q(A,M) ) \Rightarrow H_{p+q}(A \ast_{\theta, \sigma} G, M), 
 \]
 where $H_{\ast}^{par}(G,- ) $ denotes the partial homology of $G$ as defined in \cite{MMAMDDHKPartialHomology}.
 
 In Section~\ref{sec:CohomolSpecSeq} we follow dual ideas to work with cohomology. The main result is 
Theorem~\ref{t_SSCH}, which states that there exists a 
third quadrant cohomological spectral sequence of the form
\[
    E_2^{p,q} = H^p_{par}(G,H^q(A,M) ) \Rightarrow H^{p+q}(A \ast_{\theta, \sigma} G, M), 
\]
 where $M$ is an arbitrary $A \ast_{\theta,\sigma} G$-bimodule
and $H^{\ast}_{par}(G,- ) $ denotes the partial cohomology of $G$ from \cite{AlAlRePartialCohomology}. This generalizes \cite[Theorem 4.1]{AlAlRePartialCohomology}.

% \bigskip
 
 We shall work mainly with algebras over fields, nevertheless, some concepts and results we shall state over a commutative ring. It will be convenient to adopt the following notation: in all what follows $K$ will denote an arbitrary commutative unital (associative) ring, whereas $\kappa$ will stand for an arbitrary field. Furthermore, $G$ will be an arbitrary group.
 
 \section{Partial projective representations of groups into 
 algebras}\label{sec:ParProjRepr}

 Partial projective group representations with values in $\kappa$-cancellative monoids were introduced in \cite{DN}. The latter are defined as follows (see \cite[Definition 2]{DN}):
 
 \begin{definition}\label{def:K-smgrp} By a 
 \textbf{$\kappa $-semigroup} we shall mean a semigroup $T$
 with $0$ together with a map $\kappa \times T \to T$ such that $\af(\beta x) = (\af \beta )x,$ $\af (xy) = (\af x ) y = x (\af y),$ $1_{\kappa } x =x$ for any $\af , \beta \in \kappa, x,y \in T$ and the zero of $\kappa$ multiplied by any element of $T$ is $0 \in T.$ A \textbf{$\kappa $-cancellative semigroup} is a $\kappa $-semigroup $T$ such that for any $\alpha, \beta \in K$ and $0\neq x\in T$ one has that $\alpha x = \beta x \; \Longrightarrow \alpha = \beta.$ 
 \end{definition} Evidently, any unital (associative) $\kappa$-algebra is a $\kappa$-cancellative (multiplicative) monoid. We shall need some notions and facts from \cite{DN} adapted to the case of algebras.
 
 \medbreak
 
 The following is a slight modification of \cite[Definition 1]{DN} (see also \cite[Definition 2]{DN}):

 \begin{definition}\label{def:ProjMonoidRepr}
 A \textbf{projective representation} of a monoid $S$ in a unital $\kappa $-algebra $R$ is a map $\Gamma: S\to R$ for which there is a function $\rho:S\times S\to \kappa $
 such that for every $x,y\in S:$ 
 \begin{equation}\label{eq:ProjMonoid1}
 \rho (x,y)=0 \Leftrightarrow \Gamma(xy) = 0,
 \end{equation}
 and
 \begin{equation}\label{eq:ProjMonoid2}
  \quad\Gamma(x)\Gamma(y) = \Gamma(xy)\rho(x,y).
 \end{equation}
 The map $\rho$ is called the \textbf{factor set of} $\Gamma .$  
 \end{definition}
 
 If we replace in Definition~\ref{def:ProjMonoidRepr} the algebra $R$ by a $\kappa $-cancellative monoid $M,$ then we obtain a definition equivalent to \cite[Definition 1]{DN}. We say that $\G$ is {\it unital} is $\G (1_G)=1_R.$
 
 Applying
 (\ref{eq:ProjMonoid2}) to the equality $\Gamma(x)[\Gamma(y)\Gamma(z)]=
 [\Gamma(x)\Gamma(y)]\Gamma(z),$ we obtain:
 \begin{equation}\label{eq:ProjMonoid3}
 \rho(x,y)\rho(xy,z)=\rho(x,yz)\rho(y,z) \text{ for all }x,y,z \in S.
 \end{equation}

 Given a monoid $S,$ denote by ${\rm RCFMat _{S\times S}}(\kappa )$ the algebra of all row and column finite matrices over $\kappa $ whose columns and rows are indexed by the elements of $S.$ By a projective matrix representation of a monoid $S$ over $\kappa $ we shall mean a projective representation of the form 
 $S\to {\rm Mat}_n(\kappa),$ $n>0,$ or $S\to {\rm RCFMat _{S\times S}}(\kappa ).$ We need the next fact from \cite{DN}.
 
 \begin{theorem}\label{teo:MonoidFactorSet}(\cite[Theorem 1]{DN})
 Let $S$ be a monoid. A map $\rho:S\times S\to \kappa $ is a factor set
 for some projective matrix representation of $S$ over $\kappa $ if and
 only if $\rho$ satisfies \eqref{eq:ProjMonoid3} and, moreover,
 \begin{equation*}
 \rho(x,y)=0 \Longleftrightarrow \rho(1,xy)=0
 \end{equation*}
 for all $x,y\in S$.
 \end{theorem}
 
 \textbf{In all what follows} we assume that the projective representations of $S$ are unital.\medbreak
 
 Notice the following:
 
 \begin{remark}\label{rem:MonoidFactorSets} It follows from \cite[Corollary 1]{DN} that the set $m(S)$ of all factor sets of projective matrix representations of a monoid $S$ over $\kappa $ coincides with the set of all factor sets of projective representations of $S$ into $\kappa $-cancellative monoids, and consequently, $m(S)$ is the same as the set of all factor sets of projective representations of $S$ into $\kappa$-algebras.
 \end{remark}

 Define a product of factor sets $\rho$ and $\sigma$ of projective matrix representations of $S$ by the pointwise multiplication:
 $\rho\sigma(x,y)=\rho(x,y)\sigma(x,y)$. It is an immediate consequence of
 Theorem~\ref{teo:MonoidFactorSet} that $\rho\sigma$ is also a factor set, and hence, $m(S)$ is a commutative semigroup.\medbreak

 The next concept is based on \cite[Definition 6, Theorem 3]{DN}.\medbreak

\begin{definition}\label{def:parproj} A \textbf{partial projective 
 representation} of a group $G$ in a unital $\kappa $-algebra $R$ is a map $\Gamma : G \to R,$ for which there is a function 
 $\s : G\times G \to \kappa,$ such that the following postulates are satisfied for all $g,h \in G:$
 \begin{equation}\label{parproj1}
 \s (g,h) = 0 \Rightarrow \G (g\m) \G (g h)=0 \;\;\; \&\;\;\; 
 \G(g h) \G(h\m) = 0, 
 \end{equation} 
 \begin{equation}\label{parproj2}
 \G (g\m) \G (g) \G (h) = \s (g,h) \G (g\m) \G (g h),
 \end{equation}
 \begin{equation}\label{parproj3}
 \G (g) \G (h) \G (h\m) = \s (g,h) \G (gh) \G (h\m),
 \end{equation} 
 \begin{equation}\label{parproj4}
 \G (1_G) = 1_R.
 \end{equation}
\end{definition} 
 
 Similarly, as in the case of projective representations of monoids, if we replace in Definition~\ref{def:parproj} the algebra $R$ by a $\kappa $-cancellative monoid $M,$ then we obtain a definition equivalent to \cite[Definition 6]{DN} (see also \cite[Theorem 3]{DN}).\medbreak
 
 Notice that taking $h=g\m$ in \eqref{parproj1} and using \eqref{parproj4} we obtain
 \begin{equation}\label{parproj5}
 \s (g,g\m) = 0 \Rightarrow \G (g)=0 \;\;\; \&\;\;\; \G (g\m)=0. 
 \end{equation} Furthermore, for any $g,h\in G:$
 \begin{equation}\label{parproj6}
 \s (g,h) = 0 \Rightarrow \G (g) \G (h)=0. 
 \end{equation}
 Indeed, \eqref{parproj2} gives $\G (g\m) \G (g) \G (h) = 0$ which implies 
 \[
   0= \G(g) \G (g\m) \G (g) \G (h) = \s (g,g\m)  \G (g) \G (h).
 \]
 This yields $\G (g) \G (h)=0$ keeping in mind \eqref{parproj5}.\medbreak
 
 It is also easily seen that 
 \begin{equation}\label{parproj7}
 \G (g\m) \G (g h) = 0 \Leftrightarrow \G (g) \G (h)=0
 \Leftrightarrow \G (gh) \G (h\m)=0, 
 \end{equation} for all $g,h\in G.$\medbreak

 Note that Definition~\ref{def:parproj} does not imply any restriction on the values of the function  $\s: G \times G \to \kappa$ on the set 
  $${\rm Null}(\G ) = \{(g,h)\in G\times G \; \mid \; \G(g) \G(h)=0\},$$
   and therefore, in general,  $\s $ is not determined uniquely by $\G.$ Moreover, we may modify the values of $\s $ on 
   ${\rm Null}(\G )$ arbitrarily, and the properties 
   \eqref{parproj1}-\eqref{parproj4} will still be satisfied. In particular, we may require   
 $$  \G (g) \G (h)=0  \Rightarrow \s (g,h) = 0, \;\;\;\;\; \forall g,h \in G,  $$ and if this also holds, then, according to \cite[p. 258]{DN}, we shall say that $\s $ is a {\it factor set} of $\G .$ Clearly, if $\s $ is a factor set,  then 
 \begin{equation}\label{parproj8}
 \s (g,h) = 0 \Leftrightarrow  \G(g)\G(h) = 0, \;\; \forall g,h\in G.
 \end{equation}   Moreover, a factor set is uniquely determined by 
 $\G.$ Furthermore, if $\s$ is a factor set for $\G ,$ then using 
 \eqref{parproj4}, \eqref{parproj7} and \eqref{parproj8} we obtain that
 \begin{equation}\label{parproj9}
  \s(g,g\m )=0 \Leftrightarrow \G (g) \G (g\m) =0 
 \Leftrightarrow \G (g) =0 \Leftrightarrow \G (g\m)  =0
 \Leftrightarrow 
  \G (g) \G (g\m) =0,  
 \end{equation} for any $g\in G.$ In addition, in view of 
 \eqref{parproj2}-\eqref{parproj4} and \eqref{parproj9} we have that
 \begin{equation}\label{parproj10}
 \s(g,g\m )=\s(g\m ,g ),\;\;\;\;\; \forall g\in G.
 \end{equation}

 As in the case of monoids, by  a {\it partial projective matrix representation} of a group $G$ over $\kappa $  we mean a partial projective representation of the form $G\to {\rm Mat}_n(\kappa),$ $n>0,$ or $G\to {\rm RCFMat _{S\times S}}(\kappa ).$\medbreak
 
  \begin{definition}\label{def:ParRepr} By a 
  \textbf{partial representation} of $G$ in a unital $\kappa $-algebra $R$ we mean a partial projective representation with factor set $\s $ such that $\s (g,h) \in \{0,1 \}$ for all $g,h \in G.$
  \end{definition} 
  It is easy to see that Definition~\ref{def:ParRepr}
  is equivalent to the definition given in \cite[pp. 508-509]{DN}.
  Obviously, the definition given in \cite{DEP} can be formulated with values in an algebra over an arbitrary $K.$\medbreak 
 
 Partial projective representations of a group $G$ are related to projective representations of Exel's monoid $S(G),$ which was defined in \cite{exe} as the semigroup  generated by the symbols $[g]$ ($g\in G$) with the defining
 relations
 \begin{eqnarray}
 [g^{-1}][g][h]&=&[g^{-1}][gh],\nonumber\\
 {}[g][h][h^{-1}]&=&[gh][h^{-1}],\nonumber\\
 {}[g][1_G]&=&[g],\nonumber
 \end{eqnarray}
 $g, h \in G.$ It follows  that $[1_G][g]=[g],$ so that $S(G)$ is a monoid. We know that $S(G)$ is an inverse monoid \cite{exe} and by \cite{KL} and \cite{Sze} it is isomorphic to the Birget-Rhodes expansion of $G.$ In addition, the algebra which controls the partial representations of $G$ is the partial group algebra $\Kpar G$ which is the semigroup algebra $KS(G)$ (see \cite[Definition 2.4]{DEP}).\medbreak
 
 Proposition 3 from \cite{DN} can be stated as follows:
 
  \begin{proposition}\label{prop:ParFactorSets}
 A  map $\sigma:G\times G\to \kappa$ is a factor set for some
 partial projective representation  $\Gamma$ of $G$ into a $\kappa $-cancellative monoid $M$ if and only if
 there is a factor set $\rho$ of  some projective representation of  $S(G)$ in $M$ such
 that:
 \begin{eqnarray}
 &1)&\  \s(g,h)=0
 \Longleftrightarrow \rho([g],[h])=0, \;\;\; \;\; \forall\, g,h\in G,  \nonumber\\
 &2)&\  \sigma(g,h)=
 \frac{\rho([g],[h])\,\rho([g^{-1}],[g][h])}{\rho([g^{-1}],[gh])},
 \;\;\; \;\; \forall\, g,h \in G \;\;\; \text{such that}\;\;\;\s (g,h)\neq 0. \nonumber\\
 \end{eqnarray}
  \end{proposition}
  
  We want to be sure that we do not lose factor sets if we restrict ourselves to partial projective representations with values in $\kappa $-algebras. This is guaranteed by the next:
 
 \begin{remark}\label{rem:ParFactorSets}
 Any factor set of a partial projective representation of the group $G$ into some $\kappa $-cancellative monoid is also a factor set of some partial projective matrix representation of $G$ over $\kappa,$ and hence into a $\kappa $-algebra. 
 \end{remark}
 
 \begin{proof}
     Indeed, let  $\s$ be a factor set of a partial projective representation of  $G$ into a $\kappa $-cancellative monoid $M.$ Then by Proposition~\ref{prop:ParFactorSets} there is a projective representation $S(G)\to M$ with factor set $\rho$ such that 1) and 2) of Proposition~\ref{prop:ParFactorSets} hold. Then by Remark~\ref{rem:MonoidFactorSets} $\rho$ is also a factor set of some projective matrix representation of $S(G).$ Using again Proposition~\ref{prop:ParFactorSets} we conclude that $\s $ is a factor set of some partial projective matrix representation of $G$  over $\kappa$.
 \end{proof}

 It is a direct consequence of Proposition~\ref{prop:ParFactorSets} that the factor sets of partial  projective representations of  a group $G$ into $\kappa $-cancellative monoids form a commutative semigroup $pm(G)$ with respect to the pointwise multiplication
 $(\sigma \tau)(g,h)=\sigma (g,h) \tau (g,h)$. Thanks to  Remark~\ref{rem:ParFactorSets} $pm(G)$ coincides with the set of all factor sets of partial  projective representations of  $G$ into
 $\kappa $-algebras and also with the set all factor sets of partial  projective matrix representations of  $G$ over
 $\kappa .$  Observe that $pm(G)$  is 
  a commutative inverse monoid \cite[p. 259]{DN}.\medbreak
  
  We proceed with the following:
  
  \begin{definition}\label{def:sigma-repr} Let $\s \in pm(G).$ By a \textbf{partial projective $\s $-representation} (or shortly a partial 
  $\s $-representation) of  $G$ we shall mean a partial projective representation $\G : G \to R $ of  $G$ in a unital  $\kappa  $-algebra $R$ such that $\G$ and $\s $ satisfy 
  \eqref{parproj1}-\eqref{parproj4}.  
  \end{definition}

  The partial $K$-representations of a group $G$ are governed by the partial group algebra $K_{par}G$ (see \cite{DEP}), we now define a $\kappa$-algebra which will control the partial projective $\sigma$-representations.

  \begin{definition}\label{def:ParTwiGrAlg}  Let $\s \in pm(G).$  The \textbf{twisted partial group algebra} of $G$  is the  (universal)  $\kappa $-algebra $\kappa_{par}^{\sigma}G$ generated by the  symbols  $[g]^\sigma, \, (g\in G),$ subject to the relations
\begin{align}
    \s(g,h)=0 & \Rightarrow [g^{-1}]^\sigma [gh]^\sigma=[gh]^\sigma  [h ^{-1}]^\sigma =0, \\
    [g\m]^\sigma [g]^\sigma [h]^\sigma &= \s (g,h) [g\m]^\sigma [gh]^\sigma, \\ 
    [g]^\sigma [h]^\sigma [h\m ]^\sigma &= \s (g,h) [gh]^\sigma  [h\m]^\sigma, \\
    [g]^\sigma [1_G]^\sigma &= [1_G]^\sigma [g]^\sigma=[g]^\sigma, 
\end{align}
 for all $g,h \in G.$
  \end{definition} 
  Obviously, $[1_G]^\sigma= 1_{\kappa_{par}^{\sigma}G}.$ Evidently, the map $g \mapsto [g]^\sigma$ is a partial $\sigma$-representation of $G$  in  $\kappa_{par}^{\sigma}G$, which will be called the {\it canonical partial $\sigma$-representation} of $G$ and denoted by $[\cdot ]^\sigma.$ In particular,
\begin{equation}\label{parTwiA}
    \s (g,g\m) = 0 \Rightarrow [g]^\sigma=0 \;\;\; \& \;\;\;  [g\m]^\sigma=0,
\end{equation}
\begin{equation}\label{parTwiB}
    \sigma(g,h) = 0 \Rightarrow [g]^\sigma [h]^\sigma=0,
\end{equation} 
and
\begin{equation}\label{parTwiC}
    [g\m]^\sigma [g h]^\sigma = 0 \Leftrightarrow [g]^\sigma [h]^\sigma=0 \Leftrightarrow [gh]^\sigma [h\m]^\sigma=0, 
\end{equation} 
for all $g,h\in G.$ \medbreak

  %The next fact is an immediate consequence of Definitions~\ref{def:parproj} and~\ref{def:ParTwiGrAlg}.
  
  If  $\G :G \to R$ is a partial projective $\s$-representation, 
  then it is an immediate consequence of Definitions~\ref{def:sigma-repr} and~\ref{def:ParTwiGrAlg} that the map $\tilde{\G} : \kpar ^\sigma G \to R,$ determined by $[g]^\sigma \mapsto \G (g),\, g\in G,$ is a well-defined $\kappa$-algebra homomorphism such that  $\tilde{\G} \circ  [\cdot ]^\sigma = \G.$  Conversely, it is obvious that since $[\cdot ]^\sigma$ is a partial $\s$-representation, then for any $\kappa$-algebra homomorphism $\tilde{\Gamma} : \kpar ^\sigma G \to R$ the map $\tilde{\Gamma}\circ [\cdot]^\sigma : G \to R$ is a partial  $\s$-representation. Thus, we have the next fact.

 \begin{proposition}\label{prop:1to1corresp} 
   Let $\s \in pm(G).$ For any partial $\s$-representation $\G $ of $G$ into a unital $\kappa $-algebra $R$ there exists a unique $\kappa $-algebra homomorphism $\tilde{\Gamma} : \kpar ^\sigma G \to R$ such that 
 \begin{equation}\label{univProperty} 
  \tilde{\Gamma} \circ  [\cdot ]^\sigma = \Gamma.
 \end{equation}
    Conversely, for any $\kappa$-algebra homomorphism $\tilde{\Gamma} : \kpar ^\sigma G \to R$ gives rise to a partial $\s$-representation $\G :G \to R$ such that \eqref{univProperty} holds.
 \end{proposition}

It is well-known that there exist an anti-isomorphism of algebras $K_{par}G \to K_{par}G$ such that $[g] \mapsto [g^{-1}]$, consequently the algebras $K_{par}G$ and $(K_{par}G)^{op}$ are isomorphic via this anti-isomorphism. We want to understand the relation between the twisted group algebra and its opposite algebra. 

\begin{definition}
    Let $\rho \in pm(G)$. We define the involution $\rho^*$ of $\rho$ such that $\rho^*(g,h)=\rho(h^{-1},g^{-1})$.
\end{definition}

\begin{lemma} \label{l_factorinverse}
    If $\rho \in pm(G)$, then $\rho^* \in pm(G)$ and there exist a surjective morphism of algebras $\kappa_{par}^{\rho^*}G \to (\kappa_{par}^{\rho}G)^{op}$ such that $[g]^{\rho^*} \mapsto [g^{-1}]^\rho \in (\kappa_{par}^{\rho}G)^{op}$. 
\end{lemma}

\begin{proof}
   Define the map $\Gamma: G \to (\kappa_{par}^{\rho}G)^{op}$ such that
        \[
            \Gamma_g=  [g^{-1}]^\rho \in (\kappa_{par}^{\rho}G)^{op}.
        \] 
    It is easy to verify that $\Gamma$ is a partial projective $\rho^*$-representation. Furthermore, $\rho^*$ is a factor set for $\Gamma$. Indeed,
    \[
        \Gamma_g\Gamma_h= 0 \Rightarrow 0 = [g^{-1}]^\rho \star_{op} [h^{-1}]^\rho = [h^{-1}]^\rho[g^{-1}]^\rho \Rightarrow 0 = \rho(h^{-1},g^{-1}) = \rho^*(g,h).
    \]
    Hence, $\rho^* \in pm(G)$. Finally, the universal property of $\kappa_{par}^{\rho^*}G$ give us the desired map. 
\end{proof}

Observe that $\rho^{**}= \rho$, thus using Lemma \ref{l_factorinverse} twice we obtain the following commutative diagram
\[\begin{tikzcd}
	{\kappa_{par}^{\rho^*}G} & {(\kappa_{par}^{\rho}G)^{op}} \\
	{(\kappa_{par}^{\rho^*}G)^{op}} & {\kappa_{par}^{\rho}G}
	\arrow[two heads, from=1-1, to=1-2]
	\arrow[two heads, from=2-2, to=2-1]
	\arrow[tail, from=1-2, to=2-2]
	\arrow[tail, from=2-1, to=1-1]
\end{tikzcd}\]
where the horizontal maps are those obtained by Lemma \ref{l_factorinverse} for $\rho$ and $\rho^*$, and the vertical maps are the canonical anti-isomorphisms. Hence, we have proved the following proposition.

\begin{proposition} \label{p_left_and_right_partial_modules_equivalence}
    The algebras $\kappa_{par}^{\rho^*}G$ and $(\kappa_{par}^{\rho}G)^{op}$ are isomorphic. Consequently, the categories $\kappa_{par}^{\rho}G$-Mod and Mod-$\kappa_{par}^{\rho^*}G$ are isomorphic.
\end{proposition}
 
 Observe that the concept of a twisted partial group algebra extends that of the partial group algebra  (see \cite[Definition 2.4]{DEP}). Indeed, a group homomorphism  of $G$ into the group of invertible elements of an algebra is clearly a partial $\s$-representation such that $\s(g,h)=1,$ for all $g,h \in G.$ Then 
  $\kpar ^\sigma G $ coincides with the partial group algebra $\kpar G$  \medbreak

 \section{Partial projective representations and unital twisted partial actions}\label{sec:PPRvsTwistedPA}

 \bigskip

 Partial projective group representations are related to twisted partial group actions and partial group cohomology  (see \cite{DKh}, \cite{DN}, \cite{DN2}, \cite{{DoNoPi}}, \cite{DoSa}).\medbreak
 
  Recall that by $K$ we denote an arbitrary unital commutative (associative) ring. For a unital (associative) $K$-algebra $A$  we denote by  
 $\U (A)$ the group of  invertible elements of $A.$ In particular, if $I$ is an ideal in $A,$ which is generated by a central idempotent $e,$ then   $\U (I)$ means the group of the elements of $I$ which are invertible in $I$ with respect of the unity element $e$ of $I.$ In the case $I=\{ 0 \}$ we obviously have  $\U (I)=\{ 0 \}.$\medbreak

 We recall from  \cite{DES2} the following:
 
 \begin{definition}\label{def:twparac} A  \textbf{unital  twisted partial action}  of  a group $G$ on a  $K$-algebra $A$  is a triple 
 $$\Theta=(\{D_g\}_{g\in G}, \{\0_g\}_{g\in G}, \{ {\varsigma} _{g,h}\}_{(g,h)\in G\times G}),$$
 such that for every $g\in G,$ $D_g$ is  an ideal of $A,$ which is a unital algebra with unity element  $1_g,$ $\0 _g \colon D_{g\m}\to D_g$
 is a $K$-algebra isomorphism, and for each $(g,h) \in G\times G,$ ${\varsigma}_{g,h} \in \U (D_g D_{gh}),$ satisfying the following postulates\footnote{Notice that the equality $D_g=0$ is not prohibited.} for all  $g,h,t\in G:$
 \begin{enumerate} [(i)]
    \item $D_1=A$ and $\0_1$ is the identity map $A \to A,$
    \item $\0 _g(D_{g\m}D_h)=D_gD_{gh},$ 
    \item $\0_g \circ \0 _h(a)=\varsigma _{g,h}\0_{gh}(a) \varsigma \m _{g,h}$ for any $a\in D_{h\m} D_{(gh)\m},$
    \item $\varsigma_{1,g}=\varsigma_{g,1}=1_g$  and
    \item $\0_g(1_{g\m}\varsigma _{h,t})\varsigma _{g,ht}= \varsigma _{g,h} \varsigma _{gh,t}.$
 \end{enumerate}
 \end{definition}

 % Using (v) one obtains
 %\begin{equation}\label{afom}\0_g(\s_{g\m,g})=\s_{g,g\m},\,\, \,\,%\text{for any}\,\,g\in G.\end{equation}
 
 Given a unital twisted partial action $\Theta$  of $G$ on $A,$ 
  the {\it crossed product}  $A \ast_{\Theta} G$ is the direct sum:
 $$
 \bigoplus_{g \in G} D_g \de_g,
 $$ 
 in which the $\de_g$'s are placeholder symbols. The multiplication is defined by the rule: 
 $$
 (a_g {\de}_g) \cdot (b_h {\de}_h) =
  a_g {\0}_g (1_{g\m} b_h) \varsigma _{g,h} {\de}_{gh}.
 $$ 
 By \cite[Theorem 2.4]{DES1}, $A \ast _{\Theta}G$ is an associative unital $K$-algebra. If $\varsigma _{g,h} = 1_g 1_{gh}$ for all $g,h\in G,$ then we say that the twist $\varsigma $ is trivial and $\Theta $ may be identified with the family of isomorphisms $\0=\{\0_g\}_{g\in G}$ which fits the following definition:
 
 \begin{definition}\label{def:parac} A  \textbf{unital  partial action}  of  a group $G$ on an  $K$-algebra $A$ is a family of $K$-algebra isomorphisms 
 $$\0 = \{ \0 _g \colon D_{g\m}\to D_g, g\in G \},$$ where 
 each $D_g,$ $g\in G,$ is  an ideal of $A,$ generated by a central idempotent  $1_g$ of $A,$ such that  the following properties are satisfied   for all  $g,h,t\in G:$
 
 \begin{enumerate}[(i)]
    \item $D_1=A$ and $\0_1$ is the identity map $A \to A,$
    \item $\0 _g(D_{g\m}D_h)=D_gD_{gh},$
 \end{enumerate}
 \begin{enumerate}[(i')]
    \setcounter{enumi}{2}
    \item $\0_g \circ \0 _h(a)=\0_{gh}(a) $ for any $a \in D_{h\m} D_{(gh)\m}.$ \label{e_iii'}
 \end{enumerate}
 \end{definition}
 
  %and $\0$  is called  a {\it unital partial action} of $G$ on $A.$ %In this case, $(ii)$ of Definition~\ref{def:twparac} implies that
 % $\0 _{gh}(t) \in D_g D_{gh}$ if $t\in D_{h\m} D_{(gh)\m},$ and  %item $(iii)$ of  Definition~\ref{def:twparac} takes the form:\\
 
 %\noindent $(iii')$ $\0_g \circ \0 _h(a)=\0_{gh} (a)$ for any $a\in %D_{h\m} D_{(gh)\m}.$\\
 %In this case the corresponding crossed product will be denoted by $A \rtimes _{\0}G.$
 
 Thus, a unital partial action $\0 $ is a particular case of a twisted partial action. In this case, the corresponding crossed product is denoted by $A \rtimes _{\0}G$ and called the {\it partial skew group algebra}.\medbreak
 
 Obviously $1_g 1_{gh}$ is the unity element of $D_gD_{gh}$ and $(ii)$ of Definition~\ref{def:parac} implies
 \begin{equation}\label{eq:unities}
 \0_g(1_{g\m} 1_h) = 1_g 1_{gh},
 \end{equation} for all $g,h\in G.$\medbreak
 
 We are interested in the important case of a twisted partial action $\Theta$ in which  $\varsigma _{g,h}\in \U(K) 1_g 1_{gh},$ for all $g,h \in G.$  Then each $\varsigma _{g,h}$ is central in $A,$ $(iii)$ of  Definition~\ref{def:twparac} is of the form $\ref{e_iii'}$ and $\varsigma $ can be seen as a partial $2$-cocycle in the sense of \cite[Definition 1.6]{DKh}. Indeed, let $\tilde{A}$ be the $K$-subalgebra of $A$ generated by all $1_g, g\in G,$ and let $\tilde{\0}$ be the unital partial action of $G$ on $\tilde{A}$ obtained  restricting each $\0_g$ to $K1_{g\m}.$ Obviously,  $\tilde{A}$ is contained in the center  of $A$   and  considering $\tilde{A}$  as a commutative (multiplicative) monoid, $\tilde{A}$   becomes  a unital partial $G$-module as defined in \cite[Definition 1.1]{DKh}, and $\varsigma $ is an element of the group $Z^2(G,\tilde{A})$ of the partial $2$-cocycles of $G$ with values in the partial $G$-module $\tilde{A}$ (see \cite[p. 148]{DKh}). In this case we write $\Theta = (\0 , \varsigma),$ i.e.  $\Theta$ is composed by the unital partial action $\0$ and the partial $2$-cocycle   $\varsigma $ related to $\0.$\medbreak 
   
\begin{definition}\label{def:par2-cocycle}  
    Let $\0$ be a unital  partial action of  a group $G$ on a  $K$-algebra $A.$ By a \textbf{$K$-based partial $2$-cocycle} of $G$ related to $\0$ (or a \textbf{$K$-based twist} of $\theta$) we mean a function $\varsigma :  G\times G \to A, $ such that $\varsigma (g,h)=\varsigma _{g,h}\in \U(K) 1_g 1_{gh}$ for every $g,h\in G,$ and  $(\0 , \varsigma)$ is a unital twisted partial action of $G$ on $A.$
 \end{definition}
 
 The set of the $K$-based partial $2$-cocycles of $G$ related to a unital partial action $\0$ of $G$ on $A$ will be denoted by $Z^2(G, A, K).$ Clearly, $Z^2(G, A, K)= Z^2(G,\tilde{A})$ and $Z^2(G, A, K)$ is an abelian group with respect to the pointwise multiplication (see \cite[pp. 146-147]{DKh}).\medbreak

 Let $\varsigma$ be a $K$-based partial $2$-cocycle  related to a unital partial action $\0 $ of $G$ on $A.$  Write $\varsigma _{g,h} = \s (g,h) 1_g 1_{gh},$ with $\s (g,h)\in K, $ $ g,h \in G.$ We may assume that $\s (g,h) =0 $ if $1_g 1_{gh}=0.$  Then  $\s: G\times G \to K$ is a function such that for all $g,h,t\in G:$
\begin{equation}\label{domain}
  \s(g,h)=0 \Longleftrightarrow 1_g 1_{gh} = 0,
\end{equation}
 % \;\;\;\;\; \text{and}
  %\;\;\;\;\; \s(g,h) \in \U(K) \Longleftrightarrow  1_g 1_{gh} \neq %0,
\begin{equation}\label{par2-cocycle}
  1_g 1_{gh}1_{ght} \neq 0  \;\;\; \Longrightarrow \;\;\; \s(g,h) \s(gh,t) = \s(g,ht) \s(h,t).
\end{equation} 
The latter property directly follows from item $(v)$ of Definition~\ref{def:twparac}, whereas the former one is obvious. Taking $h=g\m$ in \eqref{domain} and using the evident equality $\theta_{g} (1_{g\m})=1_{g},$ we obtain, for all $g\in G,$  that
\begin{equation*}\label{1_g=0}
  1_g=0  \Longleftrightarrow  \s(g,g\m)=0 \Longleftrightarrow \s(g\m ,g)=0,
\end{equation*}
Furthermore, from $(iv)$ of Definition~\ref{def:twparac}  and taking the triple $(g\m, g, g\m)$ in \eqref{par2-cocycle} we obtain, for all $g\in G,$ that
\begin{equation}\label{1_g}
  1_g\neq 0  \Rightarrow \s(g,g\m)= \s(g\m,g) \;\; \& \;\;  \s(g,1)= \s(1,g) = 1,
\end{equation} 
so that
\begin{equation*}\label{eq:sigma}
   \s(g,g\m)= \s(g\m,g), \;\; \;\; \forall g\in G.
\end{equation*}
We call $\s$ the {\it $K$-valued twist} of $\0 $ which corresponds to the partial $2$-cocycle $\varsigma .$\medbreak

In what follows by a {\it $K$-valued partial twist} for $G$ we shall mean the  $K$-valued twist of some unital partial action $\0 $ of the group $G$ on an algebra $A$ which corresponds to some partial $2$-cocycle from $ Z^2(G, A, K).$ Note that we use the same symbol $\s $ for a $K$-valued partial twist and for a factor set. We shall see below that any element of $pm(G)$ is a $\kappa $-valued partial twist and vice versa.\medbreak
    
We proceed by recalling from \cite{DN} the construction of a unital twisted partial action from a given partial $\s$-representation.\medbreak

Let $\G :  G \to R$ be a partial projective representation in a unital $\kappa $-algebra $R$ with factor set  $\s .$ As in  \cite[Section 7]{DN} set  
\begin{equation}\label{idempotent e}
 e_g^\Gamma = \left \{  \begin{array}{ll}
        \G (g) \G (g\m)  \s (g\m , g )\m & {\rm if \ } \s(g\m, g) \neq 0, \\
 0 & {\rm if \ } \s(g\m, g) = 0.
 \end{array}  \right.
 \end{equation} 
 Since $R$ is a $\kappa $-cancellative monoid, the computations made in \cite[Section 7]{DN} are valid in our case. In particular,  it is shown in \cite[Section 7]{DN} that 
 \begin{equation}\label{Commuting e}
  (e_g^\Gamma)^2 = e_g^\Gamma, \;\;  e_g^\Gamma e_h^\Gamma = e_h^\Gamma e_g^\Gamma, \;\; 
  \Gamma(g) e_h^\Gamma = e_{gh}^\Gamma \Gamma(g),
  \;\; e_h^\Gamma \Gamma(g) = \Gamma (g) e^\Gamma_{g\m h},  \;\;\; \forall g,h \in G.
 \end{equation} 
 It follows from \eqref{parproj9}, \eqref{parproj10} and \eqref{idempotent e} that 
 \[
    \s (g\m , g)=0  \Leftrightarrow e^\Gamma_g =0 \Leftrightarrow \Gamma(g)=0 \Leftrightarrow \G (g\m)=0 \Leftrightarrow e^\Gamma_{g\m}=0 \Leftrightarrow \sigma (g , g\m)=0,   
 \] 
 for all $g\in G.$ Let $B^{\Gamma}$ be the $\kappa$-subalgebra of $R$ generated by all $e_g^\Gamma, \, (g\in G).$ In view of \eqref{Commuting e} $B^{\Gamma}$ is a commutative unital algebra. For a fixed $g\in G$ let $D^{\Gamma}_g = e_g^\Gamma B^{\G}.$ 
 %It also shown in \cite[Section 7]{DN} that
 %\begin{equation*}
 %D^{\G}_g D^{\G}_{gh}=0 \Leftrightarrow \s (g,h)=0.
 %\end{equation*} 
 By the proof of \cite[Theorem 6]{DN}
 we have the following:
 \begin{proposition}\label{prop:ParAcFromParProjRep}
  The maps  $\0^{\G} _g: D^{\G}_{g\m } \to D^{\G}_g, \, g \in G,$ defined by
 \[
    \0 ^{\G} _g(b)  = \left \{  \begin{array}{ll}
        \G (g)b \G (g\m ) \s (g\m , g)\m & {\rm if \ }  e_{g\m }^\Gamma \neq 0, \\
        0 & {\rm if \ }  e_{g\m }^\Gamma = 0,
        \end{array}  \right. 
 \]
 $b\in D^{\G}_{g\m },$ form a unital  partial action $\0 ^{\G} $ of $G$ on  $B^{\G},$ and   
 $\s $ is a $\kappa $-valued twist of $\0 ^{\G}.$ 
 \end{proposition}
 
 It immediately follows that any $\s \in pm(G)$ is a $\kappa $-valued partial twist. As we mentioned already, the converse is also true and this is a direct consequence of the next fact, which gives an important example of a partial $\s$-representation.
  
  \begin{proposition}\label{prop:ImportantParSigmaRep} (see also \cite[Theorem 8]{DN}). Let $\0$ be a unital partial action  of $G$ on a $\kappa $-algebra $A,$  $ \varsigma \in Z^2(G,A,\kappa )$  and $\s$ be the corresponding $\kappa $-valued twist. Then the map $\G_\sigma : G \to A \ast _{\0,\varsigma } G,$ given by $g \mapsto 1_g \delta _g,$ is a partial projective representation and $\s$ is a factor set for $\G_\sigma.$ In particular, $\s \in pm(G).$
  \end{proposition}
\begin{proof}
    Notice  that
     \[
       \G_\sigma (g) \G_\sigma(h )= 1_g \delta _g 1_{h} \delta _{h}= \s(g,h)1_g1_{gh} \delta _{gh},
     \]so that
 $\G_\sigma (g) \G_\sigma(h ) \Leftrightarrow \s(g,h)=0. $ Moreover,
 \[
    \G_\sigma (g\m)\G_\sigma (g) \G_\sigma(h)=1_{g\m} \delta_{g\m} \s(g,h)1_g1_{gh} \delta _{gh} = \s(g,h)\s(g\m,gh)1_{g\m}1_{h} \delta _{h}   
 \]
 and
 \[
    \s(g,h)\G_\sigma (g\m)\G_\sigma (gh )= \s(g,h)1_{g\m} \delta _{g\m} 1_{gh} \delta_{gh} = \s(g,h)\s(g\m,gh)1_{g\m}1_{h} \delta_{h},   
 \]
 proving \eqref{parproj2}. The equality \eqref{parproj3} is shown similarly. If $\sigma (g,h)=0$ then $1_g 1_{gh}=0,$ which by \eqref{eq:unities} implies $1_{g\m}1_{h}=0,$ and the above calculation shows that $\G_\sigma (g\m)\G_\sigma (gh )=0.$ The second equality in \eqref{parproj1} is seen analogously.
\end{proof}
    
Now Proposition~\ref{prop:ImportantParSigmaRep} immediately gives us the next:
\begin{corollary}\label{cor:hom}  
   The map $\tilde{\Gamma}_\sigma : \kappa_{par}^{\sigma} G \to A \ast _{\0,\varsigma } G,$ such that  $[g]^\sigma \mapsto 1_g \delta_g,$  $g \in G,$ is a homomorphism of $\kappa$-algebras.
\end{corollary}
  
Corollary~\ref{cor:hom} implies, in particular, that, for all $g\in G,$ we have that $1_g \neq 0  \Rightarrow [g]^\sigma \neq 0.$ Furthermore, if $\s(g,h) \neq 0,$ then $1_g 1_{gh} \neq 0$ and
 \[
 [g]^\sigma [h]^\sigma \mapsto 1_g \delta _g \; 1_h \delta _h =
 1_g \0_g (1_{g\m}1_h) \delta _{gh} = 1_g 1_{gh} \delta _{gh} \neq 0.
 \] 
 This yields that $[g]^\sigma [h]^\sigma \neq 0.$ In addition,
 $1_g 1_{gh} \neq 0$ implies $0\neq 0_{g\m}(1_g 1_{gh})
 = 1_{g\m} 1_{h}.$ Then
 \[
 [g\m]^\sigma [gh]^\sigma \mapsto 1_{g\m} \delta _{g\m} \; 1_{gh} \delta _{gh} =
 1_{g\m} \0_{g\m} (1_{g}1_{gh}) \delta _{h} = 1_{g\m}  1_{h} \delta _{h} \neq 0.
 \]
 Hence, $[g\m]^\sigma [gh]^\sigma \neq 0.$ 
 
 In the above calculations, $\s$ was supposed to be $\kappa $-valued partial twist. However, we know already that any element of $pm(G)$  is a $\kappa $-valued partial twist and vice versa. So  recalling the defining relations of $ \kpar ^\sigma G$ we obtain for an arbitrary $\s \in pm(G)$  that
 \begin{equation}\label{parTwiD}
 \s(g,g\m)=0 \Leftrightarrow [g]^\sigma =0 \Leftrightarrow [g\m]^\sigma =0,\;\;\; \forall g\in G,
 \end{equation}  
 and 
 \begin{equation}\label{parTwiE}
 \s(g,h)=0 \Leftrightarrow [g]^\sigma[h]^\sigma =0 \Leftrightarrow [g\m]^\sigma[gh]^\sigma =0 \Leftrightarrow [gh]^\sigma[h\m]^\sigma =0, \;\;\;  \forall g,h\in G.
 \end{equation} 
 In particular,  \eqref{parTwiE} implies the following:
 
\begin{corollary}\label{cor:FactorSetFor[]} 
   For any $\s \in pm(G)$ we have that $\s $  is the factor set of the canonical partial  $\s$-representation $[\cdot]^\sigma : G \to \kpar ^\sigma G .$
\end{corollary} 
Observe that the condition $\G (1_G) = 1_R \neq 0$ guarantees that $\kpar ^\sigma G \neq 0.$ Indeed, taking 
  $g=h=1_G$ in \eqref{parproj2}  we obtain  
   $\s(1_G, 1_G)=1$ which gives $1_{\kappa_{par}^{\sigma}G} = [1_G]^\sigma \neq 0$ in view of \eqref{parTwiD}.\medbreak
  
  Propositions~\ref{prop:ParAcFromParProjRep} and \ref{prop:ImportantParSigmaRep} imply:
  
\begin{corollary}\label{cor:FactorSets=Twists}
  The semigroup $pm(G)$ coincides the set of all $\kappa $-valued partial twists for $G.$ 
\end{corollary}

   We know that  $\Kpar G$ is the skew group algebra by a unital partial action of $G$ (see \cite[Theorem 6.9]{DE1} or \cite[Theorem 10.9]{E6}).  We are going to show that  it is possible to  endow 
   $\kpar ^\sigma G$ with the structure of a  crossed product by a unital twisted partial action of $G.$ It will be helpful for us to introduce a  twisted version of the useful notion of a covariant representation of a partial action (see \cite[Definition 9.10]{E6}).

\begin{definition}\label{def:Covariant} Let  $\s$ be a $\kappa $-valued twist
   of a unital partial action 
   $$\0 = \{\0_g: D_{g\m} \to D_g \}_{g \in G } $$ of a group $G$ on an algebra $A.$
   A \textbf{covariant $\s$-representation} of $\0$ in a unital algebra $R$ is a 
   pair $(\pi, \Gamma),$ where $\pi : A \to R$ is a homomorphism of $K$-algebras and $\G : G \to R $ is a partial $\s$-representation 
   such that for each $g\in G$  the equality
   $$\G (g) \pi (a) \G (g\m)  = \s (g,g\m) \pi (\0_g(a)),$$
   holds for all $ a \in D_{g\m}.$
   \end{definition}
  A  covariant $\s$-representation is a tool to construct representations of the crossed product by a unital partial action twisted by $\s $ as shown in the next:
  
  \begin{proposition}\label{prop:covariant}  Let  $\s$ be a $\kappa $-valued twist
   of a unital partial action $\0 $ of a group $G$ on an algebra $A$
   and $(\pi, \Gamma)$ a  covariant $\s$-representation of $G$  in a unital algebra $R.$  Then the map $\pi \times \G : A\ast _{\0, \s} G \to R$ determined by 
  $$a \delta _g \mapsto \pi (a) \G(g), \;\;\; \forall g\in G, \;\;\; \forall a\in D_{g}, $$
  is a homomorphism of $\kappa $-algebras.
  \end{proposition}
  
  \begin{proof}
     Notice that
  \begin{equation}\label{PiGamma}
  \pi(c) \G(g\m) \G(g) = \s(g\m,g)  \pi (c),
  \end{equation} for all $g\in G$  and $c \in D_{g\m}.$ Indeed, the equality is trivial if  $c=0,$ so  assume $c\neq 0$. In particular, $1_{g\m} \neq 0.$ Then
  \begin{align*}
  &\pi(c) \G(g\m) \G(g) =  
  \pi(\0 _{g\m} (\0_{g}(c))) \G(g\m) \G(g) \\
  & =\s(g\m,g)\m \s(g,g\m)\m \G(g\m) \G(g) \pi (c)  
  \underbrace{\G(g\m) \G(g) \G(g\m)} \G(g)\\
  & =\s(g\m,g)\m \s(g,g\m)\m \s (g\m,g) \G(g\m) \G(g) \pi (c)  
  \G(g\m) \G(g) \\
  & =\s(g\m,g) \pi(\0 _{g\m} (\0_{g}(c)))  
   = \s(g\m,g)  \pi (c),
   \end{align*} as desired. Furthermore,
  \begin{equation}\label{PiGamma2}
  \pi(1_{g})  \G(g) = \G(g)  \pi (1_{g\m}), \text{ for all } g \in G \text{ with } 1_g \neq 0.
  \end{equation} 
   To see this it is enough to  multiply the equality 
   $\pi (1_g) = \G(g) \pi (1_{g\m}) \G(g\m) \s(g,g\m)\m$ by $\G (g)$ to obtain
   $$ \pi (1_g) \G (g) = \s(g,g\m)\m \G(g) \pi (1_{g\m}) \G(g\m)\G(g) \overset{(\ref{PiGamma})}{=}  \G(g) \pi (1_{g\m}) .$$ 
   
   In order to show that $(\pi \times \G)$ is multiplicative take    $g,h\in G,$   $a\in D_g, $ $b \in D_h.$ If  $1_g$ or $ 1_h$ equals $0$, then $a=0$ or $b=0$ and   
  $(\pi \times \G) (a \delta _g) (\pi \times \G) (b \delta _h)
  =0= (\pi \times \G) (a \delta _g \;  b \delta _h).$ Thus, we may assume  that $1_g \neq 0$ and $1_h \neq 0.$  Then
   \begin{align*}
  &  (\pi \times \G) (a \delta _g \;  b \delta _h)=
  (\pi \times \G) (\s(g,h) a  \0_g (1_{g\m} b) \delta _{gh})\\
  &=\s(g,h) \pi (a \0_g (1_{g\m} b)) \G({gh})=
  \s(g,h) \s(g,g\m)\m \pi (a)  \G(g) \pi( b 1_{g\m}) \underbrace{\G(g\m) \G (gh)}\\
  &=\s(g,h)  \s(g,g\m)\m \s(g,h)\m \pi (a)  \G(g) \pi(1_{g\m} b)
   \G(g\m) \G (g) \G(h)\overset{(\ref{PiGamma})}{=} \\
    &  \s(g,g\m)\m \s(g\m,g ) \pi (a)  \G(g) \pi(1_{g\m} b)
    \G(h)\overset{(\ref{1_g})}{=}\\
  &\pi (a) \G(g) \pi (1_{g\m}) \pi(b) \G(h) \overset{(\ref{PiGamma2})}{=}
   \pi (a 1_g) \G(g)  \pi(b) \G(h)=   
  (\pi \times \G) (a \delta _g) (\pi \times \G) (b \delta _h).
   \end{align*} 
  \end{proof}

 We are now ready to give a partial crossed product structure on 
 $\kpar ^\sigma G .$

  \begin{theorem}\label{prop:kparCrossed}  Let  $\s \in pm(G).$  Then there exists a unital partial action 
  $$ \0^\sigma = \{\0^\sigma_g :  D^\sigma_{g\m } \to  D^\sigma_{g } \}_{g\in G}$$ of $G$ on a commutative subalgebra  $B^\sigma$ of $\kpar ^\sigma G ,$ generated by idempotents $e_g^\sigma, (g\in G),$ with $D^\sigma_{g } = e_g^\sigma B^\sigma,$ for which $\s$ is a $\kappa $ -valued twist, such that the  map 
\[ 
  \Phi ^\sigma : \kpar ^\sigma G\to B^\sigma \ast _{(\0^\sigma , \s ) } G, \;\; \text{determined by}\;\; [g]^\sigma \mapsto e_g^\sigma \delta _g,\;\; (g\in G),
\] 
is an isomorphism of $\kappa $-algebras, whose inverse is given by
\[
  \Psi ^\sigma :  B^\sigma \ast _{(\0^\sigma , \s )} G \ni b \delta _g \mapsto b [g]^\sigma \in \kpar ^\sigma G.
\]
\end{theorem}
 
  \begin{proof}
     Taking $\G = [\cdot]^\sigma$ in \eqref{idempotent e} we obtain the idempotents $e_g^\sigma:=e_g^{\Gamma}$ and the commutative subalgebra  $B^\sigma :=  B^{\G}$ of $\kpar ^\sigma G.$ Then Proposition~\ref{prop:ParAcFromParProjRep} gives us the unital partial action   $\0^\sigma = \0^{\G}$ with  $ D^\sigma_{g } =  D^{\G}_{g } = e_g^\sigma B^{\G}.$  By Proposition~\ref{prop:ImportantParSigmaRep} the map 
  $G\to  B^\sigma \ast _{(\0^\sigma , \s )} G $ given by 
  $g\mapsto e_g\delta _g,$ $(g\in G),$ is a partial $\s$-representation, which, thanks to Proposition~\ref{prop:1to1corresp},  gives rise to the $\kappa $-algebra homomorphism 
  $\Phi ^\sigma,$ as announced in the Theorem.
  
  Observe now that it is an immediate consequence of the definition of $\0^\sigma$ and \eqref{parproj10} that the  identity map $ B^\sigma  \to B^\sigma$ and the canonical partial representation 
  $[\cdot]^\sigma$ form a covariant $\s$-representation of $\0^\sigma$ in 
  $\kpar ^\sigma G.$ Then by Proposition~\ref{prop:covariant} we obtain the  $\kappa $-algebra homomorphism  $\Psi^\sigma$  stated in our Theorem.
  
  We shall show that $\Psi^\sigma$ and  $ \Phi^\sigma$ are inverse to each other. On the one hand, for all $g\in G,$ we have that 
  \[
  \Psi^\sigma (\Phi^\sigma ([g]^\sigma)) = \Psi^\sigma (e_g^\sigma \delta _g)
  = e_g^\sigma [g]^\sigma = [g]^\sigma, 
  \]
  in view of \eqref{Commuting e}, so that $\Psi^\sigma \circ \Phi^\sigma = {\rm id}_{\kpar ^\sigma G}.$ On the other hand,  it is clearly enough to check the effect of $\Phi^\sigma \circ \Psi^\sigma$ on an element of the form $c=e_{h}^\sigma\ldots e_{t}^\sigma e_g ^\sigma \delta _g,$ $(h, \ldots t,   g \in G).$
  If $\s(x,x\m)=0 $ for some $x\in \{h, \ldots t,g\}$ then 
  $e_x=0$ so that $c=0$  and $\Phi^\sigma (\Psi^\sigma (c))=
  \Phi^\sigma (0)=0 = c.$ Assume that $\s(x,x\m)\neq 0 $ for each  $x\in \{h, \ldots t,g\}.$ Then
  \begin{align*}
&\Phi^\sigma (\Psi^\sigma (e_{h}^\sigma\ldots e_{t}^\sigma e_g^\sigma \delta _g)) 
    =\Phi^\sigma (e_{h}^\sigma \ldots e_{t}^\sigma e_g^\sigma [g]^\sigma)\\
   &=\s(h,h\m)\m \ldots \s(t,t\m)\m  \Phi^\sigma ([h]^\sigma[h\m]^\sigma \ldots [t]^\sigma[t\m]^\sigma [g]^\sigma)\\
   &= \s(h,h\m)\m \ldots \s(t,t\m)\m e_{h}^\sigma\delta_h \, e_{h\m}^\sigma\delta_{h\m} \, \ldots e_{t}^\sigma \delta_{t}
   \, e_{t\m}^\sigma \delta_{t\m}\,  e_g^\sigma \delta_g \\
   &=\s(h,h\m)\m \ldots \s(t,t\m)\m  \s(h,h\m) e_h^\sigma \delta_1 \ldots \s(t,t\m)e_t^\sigma \delta_1 e_g^\sigma \delta_g \\
   &= e_h^\sigma  \ldots e_t^\sigma e_g^\sigma \delta_g,
   \end{align*} 
   and thus $\Phi^\sigma \circ \Psi^\sigma = {\rm id}_{ B^\sigma \ast _{(\0^\sigma , \s )} G},$ completing our proof.
  \end{proof}

   \vspace*{10mm}
   
   \section{Remarks on bimodules related to projective partial representations}\label{sec:Remarks}

 In what follows  $A$ will denote a unital algebra over a field  $\kappa $   and $\Theta =(\0,  \varsigma )$ will be a unital twisted partial action of a group $G$ on $A$ such that $\0$ is a unital partial action of $G$ on $A$ and $\varsigma $ is a $\kappa $-based twist of $\0,$ i.e. $\varsigma \in Z^2(G, A, \kappa ).$ As above, $\s$ will be the $\kappa $-valued partial twist corresponding to $\varsigma .$ Note that $\s \in pm(G)$  by Corollary~\ref{cor:FactorSets=Twists}. All algebras will be unital with $1 \neq 0.$
 
 \medbreak
 
 We proceed with the following:

\begin{lemma}\label{lemma:ProjParRepBimodule}
 Let $\G  : G \to R$  be a partial $\s$-representation of  $G$ in a unital $\kappa $-algebra $R.$ If $X$ is a unital $R$-bimodule then  the map 
 $$
 \G ' : G \to {\rm End}_K(X)
 $$ 
 given by
 $$
 \G ' _g (x) : = \G (g) \cdot x \cdot \G (g\m ), \;\;\;\;\; 
 g\in G, x\in X,
 $$
 is a partial $ \s ' $-representation, where
 $$
 \s ' (g,h) = \s (g,h) \s(h\m, g\m),  \;\;\;\;\;  g,h\in G.
 $$ 
 \end{lemma}
 
 \begin{proof}
Since $\s \in pm(G),$ by Lemma \ref{l_factorinverse} we have that $\sigma^* \in pm(G)$. Hence, $\s' = \sigma \sigma^* \in pm(G),$ as $pm(G)$ is multiplicatively closed.
 
Take $g,h \in G$ and $x \in X.$ Then
 \begin{align*}
 & \G '_ {g\m} \G '_g \G '_h (x) =  
 (\G _ {g\m}   \G _g \G _h)  \cdot x 
 \cdot ( \G _ {h\m}  \G _{g\m}   \G _g) \\
 & =\s (g,h) \s(h\m, g\m) \, (\G _ {g\m}   \G _{gh})  \cdot x 
 \cdot ( \G _ {(gh)\m}   \G _g)  = 
 \s (g,h) \s(h\m, g\m) \, \G ' _ {g\m}   \G ' _{gh}  ( x), 
 \end{align*} and similarly,
 \[
   \G '_ {g} \G '_h \G '_{h\m} (x) =  \s (g,h) \s(h\m, g\m) \, 
   \G ' _ {gh}   \G ' _{h\m}  ( x),  
 \]
 showing \eqref{parproj2} and \eqref{parproj3} for $\G '.$  If 
 $\s '(g,h)=0,$ then $\s(g,h)=0$ or $\s (h\m, g\m)=0.$ In the case $\s(g,h)=0$ we have $\G _ {g\m}   \G _{gh} =0$ by \eqref{parproj1} and consequently 
 $\G ' _ {g\m}   \G ' _{gh} =0.$  Analogously $ \G ' _ {gh}   \G ' _{h\m} =0.$ If $\s (h\m, g\m)=0$ then $\G _ {(gh)\m}   \G _g=0$ and again
 \[
   \G ' _ {g\m}   \G ' _{gh} =0 =  \G ' _ {gh}   \G ' _{h\m},
 \]
 proving \eqref{parproj1} for $\G ' .$ Finally, the equality $\G ' (1)=1$ is evident.
 \end{proof}
    
\begin{remark} \label{r_left_right_sigmaprime_modules_are_equivalent}
    Observe that $(\sigma')^*=\sigma'$, therefore by Proposition \ref{p_left_and_right_partial_modules_equivalence} we have that the categories $\kappa_{par}^{\sigma'}G$\textbf{-Mod} and \textbf{Mod-}$\kappa_{par}^{\sigma'}G$ are isomorphic. Explicitly, if $X$ is a left $\kappa_{par}^{\sigma'}G$-module, then it is a right $\kappa_{par}^{\sigma'}G$-module, with the action given by
    \[
        x \cdot [g]^{\sigma'}:= [g^{-1}]^{\sigma'} \cdot x, \, \forall x \in X \text{ and } g\in G.  
    \]
 \end{remark}

 For the rest of the paper $M$ will denote an $A \ast _{\0, \s }G$-bimodule, and $\sigma' $  will stand for the element of $pm(G)$ defined in Lemma~\ref{lemma:ProjParRepBimodule}. Recall that the generators of $\kpar G ,$ $\kpar ^{\s }G $ and $\kpar ^{\sigma' }G $ are denoted by $ [g],  [g]^{\sigma}$ and $ [g]^{\sigma'},$ respectively, where $g\in G.$ We adopt a similar notation for the idempotents  $e_g \in \kpar G ,$  $e_g^{\sigma} \in \kpar ^{\s }G $ and $e_g^{\sigma'} \in \kpar ^{\sigma' }G $ and the corresponding subalgebras $B,$ $B^{\sigma}$ and $B^{\sigma'},$ where the algebra $B$ is generated by the idempotents $e_g =[g][g\m], g\in G,$ and where  $e_g^{\sigma}$ and  $e_g^{\sigma'}$ are defined by formula  \eqref{idempotent e} used for the corresponding canonical partial representations. \medbreak

 By Proposition~\ref{prop:ImportantParSigmaRep}, the map 
 $\G_\sigma : G \to A \ast _{\0,\varsigma } G,$  
 $g \mapsto 1_g \delta _g,$ $(g\in G),$  is a partial  $\s $-representation. Then Lemma~\ref{lemma:ProjParRepBimodule} and the universal property in
 Proposition~\ref{prop:1to1corresp} gives a $\kappa$-algebra homomorphism $\kpar ^{\sigma' }G \to {\rm End}_K(M),$ which endows $M$ with the structure of a left   
  $\kpar ^{\sigma' }G $-module  determined by the following formula:
 \begin{equation} \label{e_KparG-modM}
 [g]^{\sigma'}\cdot m = ( 1_g \delta _g) \,  m  \, (1_{g\m} \delta _{g\m}),
 \;\;\;  g\in G, m\in M.
 \end{equation}

\begin{remark}\label{rem:BrightKparMod} 
Keeping in mind the canonical partial $\s$-representation, the proof of Lemma~\ref{lemma:ProjParRepBimodule} can be easily adapted to endow $B^\sigma \subseteq \kpar ^\sigma G $ with the structure of a left  $\kpar ^{\sigma' }G$-module determined by
    \begin{equation}\label{eq:BleftKparMod}
        [g]^{\sigma'} \cdot b = [g]^{\sigma}b[g^{-1}]^\sigma, \;\;\; g\in G, b \in B^\sigma.
        \end{equation} 
Thus, by Remark \ref{r_left_right_sigmaprime_modules_are_equivalent} the right $\kappa_{par}^{\sigma'}G$-module structure of $B$ is given by
    \begin{equation}\label{eq:BrightKparMod}
    b \cdot [g]^{\sigma'} = [g\m]^{\sigma}b[g]^\sigma, \;\;\; g\in G, b \in B^\sigma.
    \end{equation} 
Clearly, \eqref{eq:BleftKparMod} and \eqref{eq:BrightKparMod} endows $B$ with the structure of a left and right $\kpar G$-module, respectively, if we have $\s (g,h)=1$ for all $g,h\in G.$ 

\end{remark}

    \begin{lemma}\label{l_good_factor_set}
        Let $\rho \in pm(G)$, if $\rho(g,g^{-1}) \in \{0,1\}$ for all $g \in G$. Then,
        \begin{align}
            \rho(g,h) &= \rho(h^{-1},g^{-1})^{-1}  \label{e_equivalent_factor_set1} \\ 
            \rho(g,h) &= \rho(h, h^{-1}g^{-1}) \label{e_equivalent_factor_set2} 
        \end{align}
        for all $(g,h)$ such that $\rho(g,h)\neq 0$.
    \end{lemma}
    \begin{proof}
    Let $(g,h)$ such that $\rho(g,h)\neq 0$, then
    \[
            \rho(g,h) \neq 0 \iff e_g^\rho e_{gh}^\rho \neq 0.
    \] 
    Thus,
    \begin{align*}
        e_g^\rho e_{gh}^\rho 
        &= [g]^\rho [g^{-1}]^\rho e_{gh}^\rho
        = [g]^\rho e_{h}^\rho [g^{-1}]^\rho \\
        &= [g]^\rho[h]^\rho[h^{-1}]^\rho[g^{-1}]^\rho 
        = [g]^\rho[h]^\rho e_{h^{-1}}^\rho [h^{-1}]^\rho[g^{-1}]^\rho \\
        &= \rho(g,h)\rho(h^{-1},g^{-1})[gh]^\rho e_{h^{-1}}^\rho [h^{-1}g^{-1}]^\rho \\
        &= \rho(g,h)\rho(h^{-1},g^{-1})e_{g}^\rho e_{gh}^\rho.
    \end{align*}
    Thus, we obtain \eqref{e_equivalent_factor_set1}. Since $e_g^\rho e_{gh}^\rho e_{gh h^{-1}}^\rho \neq 0$, then 
    \[
    \rho(g,h)\rho(gh,h^{-1})= \rho(g,1)\rho(h,h^{-1})=1.
    \]
    Hence, 
    \[
        \rho(g,h) = \rho(gh,h^{-1})^{-1} \overset{\eqref{e_equivalent_factor_set1}}{=} \rho(h,h^{-1}g^{-1}).
    \]
    \end{proof}

    Observe that if the equation $\sigma(g,g) = x^2$ has a solution in $\kappa$ for all $g \in G_2:= \{g \in G : g^2 = 1\}$, we have that $\sigma$ is equivalent to a partial factor set $\nu$ of $G$ such that $\nu(g,g^{-1})=1$, for all $g \in G$. That is, there exist a map $\eta: G \to \kappa \setminus \{0\}$ such that
    \[
        \rho(g,h)\eta(gh)= \nu(g,h) \eta(g)\eta(h).
    \]
    Indeed, since $g \neq g^{-1}$ for all $g \in G_2^c$ we can split $G_2^c$ in two disjoints sets $C$ and $C'$ such that
    \begin{enumerate}[(i)]
        \item $C_2^c = C \sqcup C'$; 
        \item if $g \in C$, then $g^{-1}\notin C$ and $g^{-1} \in C'$;
        \item if $g \in C'$, then $g^{-1}\notin C'$ and $g^{-1} \in C$.
    \end{enumerate}
    Thus, $G = G_2 \sqcup C \sqcup C'$. We put
    \[
    \eta(g) = \left\{\begin{matrix}
        \sigma(g,g^{-1})^{-1} & \text{ if } g \in C \text{ and } \sigma(g,g^{-1}) \neq 0 \\ 
        \sigma(g,g)^{-1/2} & \text{ if } g \in G_2 \text{ and } \sigma(g,g) \neq 0\\ 
        1 & \text{ otherwise.}
        \end{matrix}\right.  
    \]
    Hence, we define
    \begin{equation}
        \nu(g,h)= \frac{\eta(g) \eta(h)}{\eta(gh)} \sigma(g,h).
    \end{equation}
    Therefore, $\nu$ is a factor set equivalent to $\sigma$, in particular, observe that for all $g \in G$ such that $\sigma(g,g^{-1}) \neq 0$, if $g \in G_2$ then
    \begin{equation}
        \nu(g,g)=  \frac{\eta(g) \eta(g)}{\eta(1)} \sigma(g,g)= \sigma(g,g)^{-1/2}\sigma(g,g)^{-1/2}\sigma(g,g)=1
    \end{equation}
    and if $g \notin G_2$ then
    \begin{equation}
        \nu(g,g^{-1})=  \frac{\eta(g) \eta(g^{-1})}{\eta(1)} \sigma(g,g^{-1})= \sigma(g,g^{-1})^{-1}\sigma(g,g^{-1})=1.
    \end{equation}

    From the above discussion, we obtain the following proposition.

    \begin{proposition} \label{p_factor_set_equivalence}
        Let $\rho$ be a factor set. Then,
        \begin{enumerate}[(i)]
            \item $\rho$ is equivalent to a factor set $\nu$ such that $\nu(g,g^{-1}) = 1$ for all $g \in G \setminus G_2$;
            \item if the equation $\rho(g,g) = x^2$ has a solution in $\kappa$ for all $g \in G_2$, then $\rho$ is equivalent to a factor set that satisfies the equations \eqref{e_equivalent_factor_set1} and \eqref{e_equivalent_factor_set2}.
        \end{enumerate}
    \end{proposition}

    \begin{proof}
        For $(i)$ observe that the map
        \[
            \eta(g) = \left\{\begin{matrix}
                \sigma(g,g^{-1})^{-1} & \text{ if } g \in C \text{ and } \sigma(g,g^{-1}) \neq 0 \\ 
                1 & \text{ otherwise}
                \end{matrix}\right.  
        \]
        give us the desired equivalence. For $(ii)$, in view of the above considerations, Lemma \ref{l_good_factor_set} implies that $\nu$ satisfies the desired equations.
    \end{proof}

    If $\rho$ and $\nu$ are equivalent factor sets such that $ \eta(gh)\nu(g,h)= \eta(g)\eta(h) \rho(g,h)$, it is easily verified that if $\rho$ is a twist of a unital partial action $\theta$, then $\nu$ is also a twist of $\theta ,$ and,  furthermore,
    \begin{align*}
        A *_{\theta,\nu} G &\to A *_{\theta, \rho}G, \\ 
        a_g \delta_g^{\nu} &\mapsto \eta(g)a_g \delta_g^{\rho}
    \end{align*}
    is an isomorphism of $\kappa$-algebras.

\section{Homological Spectral Sequence}\label{sec:HomolSpecSeq}

Thanks to Proposition \ref{p_factor_set_equivalence} we can assume without loss of generality that our fixed partial factor set $\sigma$ satisfies $\sigma(g,g^{-1}) \in \{0,1\}$ for all $g \in G \setminus G_2$. Thus, $\sigma'(g,g^{-1})=1$ for all $g \in G \setminus G_2$. Furthermore, $\sigma'$ is equivalent to an idempotent partial factor set $\sigma''$. Indeed, define $\xi : G \to \kappa \setminus\{0\}$ such that 
\begin{equation}
    \xi(g)  =   \left\{\begin{matrix}
        \sigma(g,g)^{-1}  & \text{ if }\sigma(g,g)\neq 0 \text{ and }  g^2=1 \\ 
         1  & \text{otherwise.}
       \end{matrix}\right.    
\end{equation}
Notice that the immediate consequence of the definition of $\xi$ is 
\begin{equation} \label{e_xiInverseInvarinace}
    \xi(g)=\xi(g^{-1}).
\end{equation}

Now we define an equivalent partial factor set of $\sigma'$:
\begin{equation}
    \sigma''(g,h)= \frac{\xi(g)\xi(h)}{\xi(gh)} \sigma'(g,h).
\end{equation}
Observe that if $g \in G\setminus G_2$ then $\sigma''(g,g^{-1}) \in \{0,1\}$. On the other hand, if $g \in G_2$ and $\sigma(g,g) \neq 0$, then
\[
  \sigma''(g,g)=\frac{\xi(g)\xi(g)}{\xi(g^2)} \sigma'(g,g)= \sigma(g,g)^{-2}\sigma(g,g)^2=1.
\]
Therefore, $\sigma''(g,g^{-1}) \in \{0,1\}$ for all $g \in G$. Thus, by Lemma \ref{l_good_factor_set} we have that $\sigma''(g,h)= \sigma''(h^{-1},g^{-1})^{-1}$. Furthermore, we have that $(\sigma'')^*=\sigma''$. Indeed, recall that $(\sigma')^*= \sigma'$, thus
\begin{align*}
    \sigma''(h^{-1},g^{-1})&= \frac{\xi(h^{-1}) \xi(g^{-1})}{\xi(h^{-1}g^{-1}) }\sigma'(h^{-1},g^{-1}) \\ 
    &= \frac{\xi(h^{-1}) \xi(g^{-1})}{\xi(h^{-1}g^{-1}) }\sigma'(g,h) \\ 
    &= \frac{\xi(h^{-1}) \xi(g^{-1})}{\xi(h^{-1}g^{-1}) } \frac{\xi(gh)}{\xi(g) \xi(h)} \sigma''(g,h) \\ 
    & \overset{\eqref{e_xiInverseInvarinace}}{=} \sigma''(g,h).
\end{align*}
Then, $\sigma''(g,h)=\sigma''(h^{-1},g^{-1})$ for all $g,h \in G$. Whence we obtain
\[
    \sigma''(g,h)=\sigma''(h^{-1},g^{-1}) = \sigma''(g,h)^{-1}, \text{ for all } g,h \in G \text{ such that }\sigma(g,h)\neq 0.  
\]
Therefore, $\sigma''(g,h) \in \{0,1\}$ for all $g,h \in G$ and $\sigma'$ is equivalent\footnote{Notice that if $G$ has no elements of order $2$ then $\sigma' = \sigma''$, and in this case $\xi(g)=1$ for all $g\in G$.} to $\sigma''$.

Since $\sigma'$ and $\sigma''$ are equivalent via $\xi$ then $\kappa_{par}^{\sigma''}G \to \kappa_{par}^{\sigma'}G$ such that
\begin{equation}\label{e_sigmaprimeisomorphism}
    [g]^{\sigma''} \mapsto \xi(g) [g]^{\sigma'}
\end{equation}
is an isomorphism of algebras. Thus, the categories $\kappa_{par}^{\sigma''}G$\textbf{-Mod} and $\kappa_{par}^{\sigma'}G$\textbf{-Mod} are isomorphic. Furthermore, observe that
\begin{equation} \label{e_xiformula}
    \xi(g)= \sigma(g,g^{-1})^{-1}=\sigma(g^{-1},g)^{-1}=\xi(g^{-1}), \text{ for all } g \in G \text{ such that }\sigma(g,g^{-1})\neq 0.
\end{equation}
In addition, for the basic generators of $B^\sigma$, we obtain
\begin{equation} \label{e_xisigma}
    e_g^{\sigma} = \xi(g)[g]^{\sigma}[g^{-1}]^{\sigma} \text{ for all } g\in G.
\end{equation}

Recall that since $\sigma''$ is idempotent then there exists a surjective morphism of algebra
\begin{equation}\label{eq:kparGepi}
\kappa_{par}G \to \kappa_{par}^{\sigma''}G,
\end{equation}
such that $[g] \mapsto [g]^{\sigma''}.$ Thus, $M$ is a $\kappa_{par}G$-module with the structure given by the formulas \eqref{e_KparG-modM} and \eqref{e_sigmaprimeisomorphism}, explicitly we have
\begin{equation}\label{KparG-mod M} 
    [g] \cdot m := [g]^{\sigma''} \cdot m = \xi(g)[g]^{\sigma'} \cdot m = \xi(g)(1_g \delta_g) \cdot m \cdot (1_{g^{-1}} \delta_{g^{-1}}), 
\end{equation}
for all $g \in G$ and $m \in M.$ Notice that we can consider  $\kappa_{par}^{\sigma''}G$\textbf{-Mod} as a full subcategory of $\kappa_{par}G$\textbf{-Mod} in view of the surjective morphism \eqref{eq:kparGepi}.

\medskip 

\textbf{Notation.} Since we are fixing the $\kappa$-algebra $A$, the group $G$, the partial action $\theta$ and the twist $\sigma$, from now on we will denote the crossed product $A *_{\theta, \sigma}G$ just by $\Lambda$.

\medskip 

Let $M$ be a $\Lambda$-bimodule. Then, $M$ is an $A$-bimodule with the actions induced by the identification of $A$ with $A \delta_1 \subseteq \Lambda$, i.e., the action is defined by 

\begin{equation} \label{e_AbiModSt}
    a \cdot m := a\delta_1 \cdot m \text{ and } m \cdot a := m \cdot a\delta_1.
\end{equation}

\begin{proposition} \label{p_AskewGmIsKparGm}
    Let $f:X \to Y$ be a map of $\Lambda$-bimodules. Then, $f$ is a morphism of $\kappa^{\sigma'}_{par}G$-modules, consequently $f$ is also a morphism of $\kappa_{par}^{\sigma''}G$-modules.
\end{proposition}

\begin{proof}
    By direct computation we obtain
    \[
        f([g]^{\sigma'}\cdot x) = f((1_g \delta_g)\cdot x \cdot (1_{g^{-1}} \delta_{g^{-1}})) = (1_g \delta_g)\cdot f(x) \cdot (1_{g^{-1}} \delta_{g^{-1}}) = [g]^{\sigma'} \cdot f(x),
    \]
    for all $x \in X$ and $g \in G$.
\end{proof}

Recall that $A$ is $\kappa_{par}G$-module with the action defined by
\begin{equation} \label{e_AisKparGmodule}
    [g] \cdot a := \theta_g(1_{g^{-1}}a)
\end{equation} (see, for example,  \cite[Example 3.5]{ALVES2015137}). 

\begin{proposition} \label{p_AeMAction} 
    $A \otimes_{A^e}M$ is a $\kappa_{par}^{\sigma''}G$-module with action 
    \begin{equation} \label{e_AeMKpGM} 
        [g]^{\sigma''}\cdot(a \otimes_{A^e}m):=[g] \cdot a \otimes_{A^e} [g] \cdot m.  
    \end{equation}
\end{proposition}
\begin{proof}
    First, we have to check that the map
\begin{align*}
    \Gamma_g^{\sigma''} : A \otimes_{A^e} M   &\to A \otimes_{A^e} M \\ 
        a \otimes_{A^e}m &\mapsto [g] \cdot a \otimes_{A^{e}} [g]\cdot m
\end{align*}
    is well-defined. Indeed, for all $b,c \in A$ we get
    \begin{align*}
        [g] \cdot ( a \cdot (b \otimes c)) &\otimes_{A^e} [g]\cdot m \\ 
        & = [g] \cdot (c a b) \otimes_{A^e} [g]^{\sigma''} \cdot m \\ 
        & = ([g] \cdot c )([g] \cdot a) ([g] \cdot b) \otimes_{A^e} (\xi(g)[g]^{\sigma'}) \cdot m \\ 
        & = ([g] \cdot a) \otimes_{A^e} ([g] \cdot b) \cdot(\xi(g)[g]^{\sigma'} \cdot m) \cdot ([g] \cdot c ) \\
        & = ([g] \cdot a) \otimes_{A^e} (\theta_g(1_{g^{-1}}b)\delta_1) \cdot(\xi(g)1_g \delta_g \cdot m \cdot 1_{g^{-1}} \delta_{g^{-1}}) \cdot (\theta_g(1_{g^{-1}}c)\delta_1) \\ 
        & = ([g] \cdot a) \otimes_{A^e} \xi(g)(\theta_g(1_{g^{-1}}b)1_g \delta_g) \cdot m \cdot (( 1_{g^{-1}} \delta_{g^{-1}})  (\theta_g(1_{g^{-1}}c)\delta_1) \\
        & = ([g] \cdot a) \otimes_{A^e} \xi(g)(1_g \delta_g)(b \delta_1) \cdot m \cdot (c \delta_1)( 1_{g^{-1}} \delta_{g^{-1}}) ) \\
        &\overset{\eqref{KparG-mod M}}{=}([g] \cdot a) \otimes_{A^e} [g] \cdot \left( (b \otimes c) \cdot m \right).
    \end{align*}
    We can define the map $\Gamma^{\sigma''}: G \to \operatorname{End}_\kappa(A \otimes_{A^e} M)$, such that $ g \mapsto \Gamma_g^{\sigma''}$. The map $\Gamma^{\sigma''}$ is a partial representation. Indeed, on one hand, we have
    \begin{align*}
        \Gamma_{g^{-1}}^{\sigma''} \Gamma_{g}^{\sigma''} \Gamma_{h}^{\sigma''} (a \otimes_{A^e}m)
        &= \left([g^{-1}][g][h] \cdot a \otimes_{A^e} [g^{-1}][g][h] \cdot m \right) \\ 
        &= \left([g^{-1}][g][h] \cdot a \otimes_{A^e} [g^{-1}]^{\sigma''}[g]^{\sigma''}[h]^{\sigma''} \cdot m \right) \\ 
        &= \left([g^{-1}][gh] \cdot a \otimes_{A^e} \sigma''(g,h)[g^{-1}]^{\sigma''}[gh]^{\sigma''} \cdot m \right) \\
        &= \left([g^{-1}][gh] \cdot a \otimes_{A^e} \sigma''(g,h)[g^{-1}][gh] \cdot m \right) \\
        &= \sigma''(g,h)\Gamma_{g^{-1}}^{\sigma''}\Gamma_{gh}^{\sigma''}(a \otimes_{A^e}m),
    \end{align*}
    and, on the other hand,
    \begin{align*}
        \Gamma_{g}^{\sigma''} \Gamma_{h}^{\sigma''} \Gamma_{h^{-1}}^{\sigma''}(a \otimes_{A^e}m)
        &= \left([g][h][h^{-1}] \cdot a \otimes_{A^e} [g]^{\sigma''}[h]^{\sigma''}[h^{-1}]^{\sigma''} \cdot m \right) \\ 
        &= \left([gh][h^{-1}] \cdot a \otimes_{A^e} \sigma''(g,h)[gh]^{\sigma''}[h^{-1}]^{\sigma''} \cdot m \right) \\ 
        &= \sigma''(g,h)\Gamma_{gh}^{\sigma''}\Gamma_{h^{-1}}^{\sigma''} (a \otimes_{A^e}m).
    \end{align*}
\end{proof}

In view of Proposition~\ref{p_AeMAction}, observe that for any map of $\Lambda$-bimodules $f:X \to Y$, by Proposition~\ref{p_AskewGmIsKparGm} $f$ is map of $\kappa_{par}^{\sigma''}G$-modules and, consequently,  
$1_A \otimes_{A^e} f : A \otimes_{A^e}X \to A \otimes_{A^e} Y$ is a map of $\kappa_{par}^{\sigma''}G$-modules. Hence,  we can define the covariant (right exact) functor
\[
  F_1(-):= A \otimes_{A^e}-: \Lambda^e\textbf{-Mod} \to \kappa_{par}^{\sigma''}G\textbf{-Mod}
\]
and by Remark \ref{rem:BrightKparMod} and the isomorphism \eqref{e_sigmaprimeisomorphism} the right exact functor
\[
    F_2(-):= B^{\sigma} \otimes_{\kappa_{par}^{\sigma''}G}-: \kappa_{par}^{\sigma''}G \textbf{-Mod} \to \kappa\textbf{-Mod}. 
\]
Recall that the Hochschild homology of $\Lambda$ with coefficients in $M$ is the left-derived functor of 
\[
    F(-):=\Lambda \otimes_{\Lambda^e}-:\Lambda^e\textbf{-Mod} \to \kappa\textbf{-Mod}.
\]

Notice that the left derived functor $L_\bullet F_1$ may compute the Hochschild homology of $A$ with coefficients in $M$, the only obstruction is that the source category of $F_1$ is $\Lambda^e\textbf{-Mod}$ instead of $A^e\textbf{-Mod}$. We are going to see that $\Lambda$ is projective as $A$-bimodule, and hence $L_\bullet F_1$ computes the Hochschild homology. To achieve that we need the following proposition

\begin{proposition} \label{l_LambdaIsAeProjective}
    $\Lambda$ is projective as left and right $A$-module.
\end{proposition}

\begin{proof}
   Since
   \[
    \Lambda = \bigoplus_{g \in G} D_g \delta_g,
   \]
   is a direct sum of $A$-bimodules. Then, it is enough to show that each $D_g\delta_g$ is projective as left and right $A$-module. Observe that $A$ can be decomposed in the following direct sum of left (right) $A$-modules.
   \[
        A = D_g \oplus (1-1_g)A.
   \]
   Thus, each $D_g$ is projective as a left (right) $A$-module. Therefore, $\Lambda$ is projective as left $A$-module, since $D_g \delta_g \cong D_g$ as left $A$-modules. On the other hand, notice that $D_{g^{-1}} \cong D_g\delta_g$ as right  $A$-modules via the map $\hat{\theta}_g: D_{g^{-1}} \ni a_{g^{-1}} \mapsto \theta_g(a_{g^{-1}})\delta_g \in D_g \delta_g$. Indeed,
   \[
    \hat{\theta}_g(a_{g^{-1}} b)= \theta_g(a_{g^{-1}} b)\delta_g = \theta_g(a_{g^{-1}})\theta_g(1_{g^{-1}}b)\delta_g = (\theta_g(a_{g^{-1}})\delta_g)(b\delta_1)= \hat{\theta}_g(a_{g^{-1}})\cdot b.
   \]
\end{proof}

Using a straightforward dual basis argument one easily obtains the next:

\begin{lemma} \label{l_RSprojectiveIsPSeProjective}
    Let $R$ and $S$ be unital $K$-algebras. Suppose that $X$ is a projective left $R$-module and $Y$ a projective right $S$-module. Then, $X \otimes_K Y$ is a projective left $R \otimes_K S^{op}$-module.
\end{lemma}

Hence, by a direct application of Lemma \ref{l_LambdaIsAeProjective} and Lemma \ref{l_RSprojectiveIsPSeProjective} we obtain the following: 

\begin{lemma} \label{l_AskewP}
    $\Lambda^e$ is a projective $A^e$-module.
\end{lemma}

Notice that the $A^e$-module structure of $\Lambda^e$ is given by the map of $\kappa$-algebras
\[
    A^e \ni a \otimes b \mapsto a\delta_1 \otimes b\delta_1 \in \Lambda^e.
\]
Thus, the following proposition will be useful

\begin{proposition}{\cite[Propostion 1.4]{DeMeyerSeparableAlgebras}} \label{p_DMSA}
    Let $R$ and $S$ be rings and $f: R \to S$ a homomorphism of rings such that $S$ is a projective $R$-module. Then, any projective $S$-module is a projective $R$-module.
\end{proposition}

From Proposition \ref{p_DMSA} and Lemma \ref{l_AskewP} we conclude that any projective $\Lambda$-bimodule is a projective $A$-bimodule. Therefore, any projective resolution of $M$ in $\operatorname{Rep}\Lambda^e$ is a projective resolution of $M$ in $\operatorname{Rep}A^e$. Thus, the left derived functor of $F_1$ computes the Hochschild homology of $A$ with coefficients in $M,$ with $M$ being seen as an $A$-bimodule, i.e.
    \begin{equation} 
            H_\bullet(A,M) \cong L_\bullet F_1(M).
    \end{equation}

\begin{lemma} \label{l_teneq}
        Let $X$ be a $\Lambda$-bimodule and $Y$ a left  $\kappa_{par}^{\sigma''}G$-module. Then
        \begin{enumerate}[(i)]
            \item  $e_g \cdot a = 1_g a$, for all $g \in G$ and $a \in A$, considering $A$ with the $\kappa_{par}G$-module structure determined by \eqref{e_AisKparGmodule};
            \item  $ e_g \cdot x = e^{\sigma''}_g \cdot x = 1_g \cdot x \cdot 1_g,$ $g \in G$ and $x \in X$, considering $X$ with the $\kappa_{par}^{\sigma'}G$-module structure given by \eqref{e_KparG-modM} and with the consequent $\kappa_{par}G$-module structure as in \eqref{KparG-mod M};
            \item  $e_g^{\sigma} \otimes_{\kappa_{par}^{\sigma''}G} y = 1_{B^{\sigma}} \otimes_{\kappa_{par}^{\sigma''}G} [g^{-1}]^{\sigma''} \cdot y = 1_{B^\sigma} \otimes_{\kappa_{par}^{\sigma''}G} e^{\sigma''}_g \cdot y =1_{B^\sigma} \otimes_{\kappa_{par}^{\sigma''}G} e_g \cdot y$, as elements of $B^\sigma \otimes_{\kappa_{par}^{\sigma''}G}Y$, for all $g \in G$ and $y \in Y,$
						where the left $\kappa_{par}G$-action on  $Y$ is given via \eqref{eq:kparGepi};
            \item $e_g \cdot (a \otimes_{A^e} x) = a \otimes_{A^e} e_g \cdot x = e_g \cdot a \otimes_{A^e}  x$, as elements of $A \otimes_{A^e}X$; so that $w \cdot (a \otimes_{A^e} x) = w \cdot a \otimes_{A^e} x = a \otimes_{A^e} w \cdot x$, for all $w \in B$, where the left $\kappa_{par}G$-actions on the involved $\kappa_{par}^{\sigma''}G$-modules are given by \eqref{eq:kparGepi}. 
        \end{enumerate}
\end{lemma}
    
    \begin{proof} \,
        \begin{enumerate}[(i)]
            \item By direct computation we have that $e_g \cdot a = \theta_g( \theta_{g^{-1}}(1_ga)) = 1_g a,$  with $g\in G$ and $a \in A.$
            \item The equality $e_g \cdot x = e^{\sigma''}_g \cdot x$ is immediate, if $\sigma(g,g^{-1})=0$ then $(ii)$ holds trivially. Suppose that $\sigma(g,g^{-1})\neq 0$, then
            \begin{align*}
                e^{\sigma''}_g \cdot x &= \xi(g)\xi(g^{-1})(1_g \delta_g)(1_{g^{-1}} \delta_{g^{-1}}) \cdot x \cdot (1_g \delta_g)(1_{g^{-1}} \delta_{g^{-1}}) \\ 
                &= \xi(g)\xi(g^{-1}) \sigma(g,g^{-1})\sigma(g,g^{-1})(1_g \delta_1) \cdot x \cdot (1_g \delta_1) \\ 
                &\overset{\eqref{e_xiformula}}{=} (1_g \delta_1) \cdot x \cdot (1_g \delta_1) \\ 
                &= 1_g \cdot x \cdot 1_g,
            \end{align*}
            where $g\in G, x \in X.$
            \item Let $g \in G$ and $y \in Y$, first observe that if $\sigma(g,g^{-1})=0$, then the equality is trivial since $e^\sigma = 0$ and $e^{\sigma''}=0$. On the other hand if $\sigma(g,g^{-1})\neq 0$, then,
            \begin{align*}
                e^{\sigma}_g \otimes_{{\kappa_{par}^{\sigma''}G}}y 
                &= \sigma(g,g^{-1})^{-1} [g]^\sigma[g^{-1}]^\sigma\otimes_{{\kappa_{par}^{\sigma''}G}} y\\
                &\overset{\eqref{e_xiformula}}{=}1_{B^\sigma} \cdot (\xi(g)[g^{-1}]^{\sigma'}) \otimes_{\kappa_{par}^{\sigma''}G} y \\
                & \overset{\eqref{e_sigmaprimeisomorphism}}{=} 1_{B^\sigma} \cdot [g^{-1}]^{\sigma''} \otimes_{\kappa_{par}^{\sigma''}G} y \\
                &= 1_{B^\sigma} \otimes_{{\kappa_{par}^{\sigma''}G}} [g^{-1}]^{\sigma''} \cdot y
            \end{align*}
            and
            \begin{align*}
                e^{\sigma}_g \otimes_{{\kappa_{par}^{\sigma''}G}} y &= e^{\sigma}_ge^{\sigma}_g \otimes_{{\kappa_{par}^{\sigma''}G}} y \\ 
                &=\sigma(g,g^{-1})^{-1}  \sigma(g,g^{-1})^{-1}[g]^\sigma[g^{-1}]^\sigma 1_{B^\sigma} [g]^\sigma[g^{-1}]^\sigma \otimes_{{\kappa_{par}^{\sigma''}G}} y\\
                &\overset{\ref{parproj10}}{=} \xi(g)\xi(g^{-1}) 1_{B^\sigma} \cdot [g]^{\sigma'}[g^{-1}]^{\sigma'} \otimes_{{\kappa_{par}^{\sigma''}G}} y\\ 
                & \overset{\eqref{e_sigmaprimeisomorphism}}{=} 1_{B^\sigma} \cdot [g]^{\sigma''}[g^{-1}]^{\sigma''} \otimes_{{\kappa_{par}^{\sigma''}G}} y\\ 
                &= 1_B \cdot e^{\sigma''}_g  \otimes_{{\kappa_{par}^{\sigma''}G}} y = 1_B \otimes_{{\kappa_{par}^{\sigma''}G}} e^{\sigma''}_g \cdot y.
            \end{align*}
            \item For $a \in A$ and $x \in X$ we obtain, using (i) and (ii), that
            \begin{align*}
                e_g \cdot (a \otimes_{A^e} x) &= e_g \cdot a \otimes_{A^e} e_g \cdot x\\ 
                &= 1_g a  \otimes_{A^e} (1_g \delta_1) \cdot x\cdot (1_g \delta_1) \\ 
                &= a \cdot (1_g \otimes 1_g) \otimes_{A^e} (1_g \delta_1) \cdot x\cdot (1_g \delta_1) \\ 
                &= a  \otimes_{A^e} (1_g \delta_1) \cdot x\cdot (1_g \delta_1) \\
                &= a \otimes_{A^e} e_g \cdot x
            \end{align*}
            and
            \begin{align*}
                e_g \cdot (a \otimes_{A^e} x) &= e_g \cdot a \otimes_{A^e} e_g \cdot x \\ 
                &= 1_g a  \otimes_{A^e} (1_g \delta_1) \cdot x\cdot (1_g \delta_1) \\
                &= 1_g a  \otimes_{A^e} (1_g \otimes 1_g) \cdot x \\
                &= 1_g a \cdot (1_g \otimes 1_g) \otimes_{A^e} x \\
                &= 1_g a  \otimes_{A^e} x\\
                &= e_g \cdot a  \otimes_{A^e} x.
            \end{align*}
        \end{enumerate}
    \end{proof}

\begin{remark} \label{r_iota}
    Recall that by Remark \ref{rem:BrightKparMod} we have that $B^\sigma$ is a left $\kappa_{par}^{\sigma'}G$-module with the structure given by
    \[
        [g]^{\sigma'} \cdot w := [g]^\sigma w [g^{-1}]^\sigma ,
    \]
    where $g\in G, w \in B^\sigma.$ Thus, by \eqref{e_sigmaprimeisomorphism} we have that
    \[
        [g]^{\sigma''} \cdot w := \xi(g) [g]^\sigma w [g^{-1}]^\sigma.
    \]
    In particular, 
    \begin{equation} \label{e_BsigmaLeftKparGModuleConsequence}
    e^{\sigma''}_g \cdot w =e^{\sigma}_gw. 
    \end{equation} 
    Indeed, if $\sigma(g,g^{-1})=0$, then $e_g^{\sigma''}=0$ and $e_g^{\sigma}=0$. On the other hand, if $\sigma(g,g^{-1})\neq 0$, then
    \begin{align*}
        e^{\sigma''}_g \cdot w  &= [g]^{\sigma''}[g^{-1}]^{\sigma''} \cdot w \\ 
                                &= \xi(g)\xi(g^{-1}) [g]^{\sigma'}[g^{-1}]^{\sigma'} \cdot w \\
                                &= \sigma(g,g^{-1})^{-1} \sigma(g^{-1},g)^{-1}[g]^{\sigma}[g^{-1}]^{\sigma}w[g]^{\sigma}[g^{-1}]^{\sigma} \\ 
                                &= e_g^{\sigma} w e_g^{\sigma} = e_g^{\sigma}w.
    \end{align*}
    Thus, observe that the map $\iota_0: \kappa_{par}^{\sigma''}G \ni x \mapsto x \cdot 1_{B^{\sigma}} \in B^\sigma$ is a morphism of left $\kappa_{par}^{\sigma''}G$-modules, in particular the restriction of $\iota_0$ to $B^{\sigma''}$ is a morphism of $\kappa$-algebras 
    \begin{equation} \label{e_iota}
        \iota: B^{\sigma''}\to B^\sigma
    \end{equation}
    such that $\iota(e_g^{\sigma''})= e_g^\sigma$. 
\end{remark}

Recall that  by Corollary \ref{cor:hom} the algebra $\Lambda$ is a left $\kappa_{par}^{\sigma}G$-module with action, given by 
   \begin{equation} \label{e_KparG-mod-ArxG}
    [h]^{\s } \cdot (a_g\delta_g) := (1_h \delta_h) (a_g\delta_g).
   \end{equation}
In particular $e_h^{\sigma} \cdot (a_g\delta_g) := 1_h a_g\delta_g$. Therefore, we have that $\Lambda$ is a left $B^{\sigma''}$-module via $\iota$, explicitly the left action is determined by
\begin{equation} \label{e_B_module_structure_of_AskewG}
    e_g^{\sigma''} \cdot a_h \delta_h := e_g^\sigma \cdot a_h \delta_h = 1_g a_h \delta_h.
\end{equation}

\begin{proposition} \label{p_ArtG-Mod-XtensorArtG}
    Let $X$ be a right $\kappa_{par}^{\sigma''}G$-module. Then, $X \otimes_{B^{\sigma''}} \Lambda$ is a $\Lambda$-bimodule with actions:
    \begin{equation}
        a_g \delta_g \cdot (x \otimes_{B^{\sigma''}} c_t \delta_t) := x \cdot [g^{-1}]^{\sigma''} \otimes_{ B^{\sigma''}} (a_g \delta_g)(c_t \delta_t)
    \end{equation}
    and
    \begin{equation} 
        (x \otimes_{B^{\sigma''}} c_t \delta_t) \cdot a_g \delta_g := x \otimes_{B^{\sigma''}} (c_t \delta_t) (a_g \delta_g).
    \end{equation}
    Thus, we obtain a functor $- \otimes_{B^{\sigma''}} \Lambda: \textbf{Mod-}\kappa_{par}^{\sigma''}G \to \Lambda^e\textbf{-Mod}$.
\end{proposition}

\begin{proof}
    The right action is well-defined since it is induced by the right multiplication in $\Lambda$. For the left action consider the map
    \[
        \tilde{\Theta}_{a_g \delta_g}: X \times \Lambda\to X \otimes_{ B^{\sigma''}} \Lambda \text{ such that } \tilde{\Theta}_{a_g \delta_g}(x, c_t\delta_t) = x \cdot [g^{-1}]^{\sigma''} \otimes_{B^{\sigma''}}(a_g \delta_g)(c_t \delta_t).
    \]
    Observe that
    \begin{align*}
        \tilde{\Theta}_{a_g \delta_g}(x \cdot{ e^{\sigma''}_h}, c_t\delta_t)
         &= (x \cdot { e^{\sigma''}_h}) \cdot {  [g^{-1}]^{\sigma''} } \otimes_{B^{\sigma''}} (a_g \delta_g)(c_t \delta_t) \\ 
                                    &= x \cdot   [g^{-1}]^{\sigma''}  e^{\sigma''}_{gh} \otimes_{B^{\sigma''}} (a_g \delta_g)(c_t \delta_t) \\ 
                                    &= x \cdot  [g^{-1}]^{\sigma''}  \otimes_{B^{\sigma''}} e^{\sigma''}_{gh}  \cdot \left((a_g \delta_g)(c_t \delta_t)\right) \\ 
                                    &= x \cdot {  [g^{-1}]^{\sigma''} } \otimes_{B^{\sigma''}} (1_{gh} a_g \delta_g)(c_t \delta_t) \\
                                    &= x \cdot {  [g^{-1}]^{\sigma''} } \otimes_{B^{\sigma''}} (a_g \delta_g)(1_h \delta_1)(c_t \delta_t) \\
                                    &= x \cdot {  [g^{-1}]^{\sigma''} } \otimes_{B^{\sigma''}} (a_g \delta_g)\left({ e^{\sigma''}_h} \cdot (c_t \delta_t)\right) \\
                                    &= \tilde{\Theta}_{a_g \delta_g}(x , { e^{\sigma''}_h} \cdot  c_t\delta_t).
    \end{align*}
    Then, the map
    \[
        \Theta_{a_g \delta_g}: X \otimes_{B^{\sigma''}} \Lambda \to X \otimes_{B^{\sigma''}} \Lambda \text{ such that } \Theta_{a_g \delta_g}(x \otimes_{B^{\sigma''}} c_t\delta_t) = x \cdot [g^{-1}] \otimes_{B^{\sigma''}} (a_g \delta_g) (c_t \delta_t)
    \]
    is well-defined, for all $g \in G$ and $a_g \in D_g$. For simplicity of notation it is convenient to consider $X$ as a right $\kappa_{par}G$-module and $\Lambda$ as a left   $B$-module via the  epimorphism \eqref{eq:kparGepi}, i.e.  the action of $[g]$ is the same as that of 
    $[g]^{\sigma''}.$ Now to verify the associativity of the left action observe that
    \begin{align*}
        \Theta_{b_h \delta_h} \Theta_{a_g \delta_g}(x \otimes_{B^{\sigma''}} c_t\delta_t) &= \Theta_{b_h \delta_h}(x \cdot [g^{-1}] \otimes_{B^{\sigma''}} (a_g \delta_g) (c_t \delta_t)) \\
        &= x \cdot [g^{-1}][h^{-1}] \otimes_{B^{\sigma''}} (b_h \delta_h)(a_g \delta_g) (c_t \delta_t) \\ 
        &= x \cdot [g^{-1}][h^{-1}] \otimes_{B^{\sigma''}} (1_h b_h \delta_h)(a_g \delta_g) (c_t \delta_t) \\ 
        &= x \cdot [g^{-1}][h^{-1}] \otimes_{B^{\sigma''}} e_h \cdot (b_h \delta_h)(a_g \delta_g) (c_t \delta_t) \\ 
        &= x \cdot [g^{-1}][h^{-1}]e_h \otimes_{B^{\sigma''}} (b_h \delta_h)(a_g \delta_g) (c_t \delta_t) \\ 
        &= x \cdot [g^{-1}h^{-1}]e_h \otimes_{B^{\sigma''}} \left((b_h \delta_h)(a_g \delta_g)\right) (c_t \delta_t) \\ 
        &= x \cdot [g^{-1}h^{-1}] \otimes_{B^{\sigma''}} e_h \cdot \left(\left((b_h \delta_h)(a_g \delta_g)\right) (c_t \delta_t)\right) \\ 
        &= x \cdot [g^{-1}h^{-1}] \otimes_{B^{\sigma''}} (b_h \delta_h)(a_g \delta_g)(c_t \delta_t) \\ 
        &= \Theta_{(b_h \delta_h )(a_g \delta_g)}(x \otimes_{B^{\sigma''}} c_t\delta_t).
    \end{align*}
    Thus, the left module action $a_g \delta_g \cdot (x \otimes_{B^{\sigma''}} c_t \delta_t)$ is well-defined. Finally, for $g,h \in G$ and $a_g \in D_g$, $b_h \in D_h$ we have
    \begin{align*}
        \left(a_g \delta_g \cdot (x \otimes_{B^{\sigma''}} c_t \delta_t)\right) \cdot b_h \delta_h &= \left(x \cdot [g^{-1}] \otimes_{B^{\sigma''}} (a_g \delta_g)(c_t \delta_t)\right) \cdot b_h \delta_h  \\ 
        &= x \cdot [g^{-1}] \otimes_{B^{\sigma''}} (a_g \delta_g)(c_t \delta_t)(b_h \delta_h) \\ 
        &= a_g \delta_g \cdot \left(x \otimes_{B^{\sigma''}} (c_t \delta_t)(b_h \delta_h)\right) \\ 
        &= a_g \delta_g \cdot \left((x \otimes_{B^{\sigma''}} c_t \delta_t) \cdot b_h \delta_h\right).
    \end{align*}
    Whence, we conclude that $X \otimes_{B^{\sigma''}} \Lambda$ is a $\Lambda$-bimodule.
\end{proof}

\begin{proposition} \label{p_inXB}
    The functors
    \[
        -\otimes_{\kappa_{par}^{\sigma''}G} (A \otimes_{A^e} - ): \textbf{Mod-}\kappa_{par}^{\sigma''}G \times \Lambda^e\textbf{-Mod} \to \textbf{Mod-}\kappa
    \]
    and
    \[
        (- \otimes_{B^{\sigma''}}\Lambda) \otimes_{\Lambda^e}-:\textbf{Mod-}\kappa_{par}^{\sigma''}G \times \Lambda^e\textbf{-Mod} \to \textbf{Mod-}\kappa
    \]  are naturally isomorphic.
\end{proposition}

\begin{proof}
    Let $X$ be a right $\kappa_{par}^{\sigma''}G$-module and $M$ a $\Lambda$-bimodule. For a fixed $x \in X$ define
    \begin{align*}
        \tilde{\gamma}_{x,M}: A \times M & \to (X \otimes_{B^{\sigma''}} \Lambda)  \otimes_{\Lambda^e}M \\ 
        (a,m) &\mapsto (x \otimes_{B^{\sigma''}} 1_A \delta_1) \otimes_{\Lambda^e} a \cdot m,
    \end{align*} 
    Observe that $\tilde{\gamma}_{x,M}$ is $A^e$-balanced. Indeed, for $d \otimes c \in A^e$ we have
    \begin{align*}
        \tilde{\gamma}_{x,M}(a\cdot(d \otimes c),m) &= \tilde{\gamma}_{x,M}(cad,m) \\ 
                                                &= (x \otimes_{B^{\sigma''}} 1_A \delta_1) \otimes_{\Lambda^e} cad \cdot m \\
                                                &= (x \otimes_{B^{\sigma''}} 1_A \delta_1)\cdot (c \delta_1) \otimes_{\Lambda^e} a\cdot (d \cdot m) \\            
                                                &= (x \otimes_{B^{\sigma''}} c \delta_1) \otimes_{\Lambda^e} a\cdot (d \cdot m) \\
                                                &= (c \delta_1)\cdot (x \otimes_{B^{\sigma''}} 1_A \delta_1) \otimes_{\Lambda^e} a\cdot (d \cdot m) \\
                                                &= (x \otimes_{B^{\sigma''}} 1_A \delta_1) \otimes_{\Lambda^e} a\cdot (d \cdot m) \cdot (c \delta_1) \\
                                                &= (x \otimes_{B^{\sigma''}} 1_A \delta_1) \otimes_{\Lambda^e} a\cdot (d \cdot m \cdot c) \\
                                                &= \tilde{\gamma}_{x,M}(a,(d \otimes c) \cdot m),
    \end{align*}
    so, $\tilde{\gamma}_{x,M}$ is $A^e$-balanced. Therefore, the following map is well-defined
    \begin{align*}
        \gamma_{x,M}: A \otimes_{A^e} M & \to (X \otimes_{B^{\sigma''}} \Lambda)  \otimes_{\Lambda^e}M \\ 
        a \otimes_{A^e} m &\mapsto (x \otimes_{B^{\sigma''}} 1_A \delta_1) \otimes_{\Lambda^e} a \cdot m,
    \end{align*} 
    Now define the function
    \begin{align*}
        \tilde{\gamma}_{(X,M)} : X \times (A\otimes_{A^e}M) & \to (X \otimes_{B^{\sigma''}} \Lambda)  \otimes_{\Lambda^e}M \\ 
        (x,a\otimes_{A^e}m) &\mapsto \gamma_{x,M}(a\otimes_{A^e}m)=(x \otimes_{B^{\sigma''}} 1_A \delta_1) \otimes_{\Lambda^e} a \cdot m.
    \end{align*}
    We want to show that  $\tilde{\gamma}_{(X,M)}$ is $\kappa_{par}^{\sigma''}G$-balanced. Recall that  the $\Lambda$-bimodule structure of $X \otimes_{B^{\sigma''}} \Lambda$ is  given by Proposition \ref{p_ArtG-Mod-XtensorArtG}.   Again, for simplicity of notation, we consider the involved 
    $\kappa_{par}^{\sigma''}G$-modules  as  $\kappa_{par}G$-modules and $B^{\sigma''}$-modules as $B$-modules, where the action of $[g]$ (respectively, $e_g$) is the same as that of $[g]^{\sigma''}$ (respectively, $e^{\sigma''}_g$). Let $h \in G$, such that $\sigma(h,h^{-1}) \neq 0$. Then,
    \begin{align*}
        \tilde{\gamma}_{(X,M)}(x \cdot [h],a\otimes_{A^e}m) 
        &= (x \cdot [h] \otimes_{B^{\sigma''}} 1_A \delta_1) \otimes_{\Lambda^e} a \cdot m \\
        &= (x \cdot [h]e_{h^{-1}} \otimes_{B^{\sigma''}} 1_A \delta_1) \otimes_{\Lambda^e} a \cdot m \\ 
        &= (x \cdot [h] \otimes_{B^{\sigma''}} e_{h^{-1}} \cdot (1_A \delta_1)) \otimes_{\Lambda^e} a \cdot m \\ 
        &= \big(x \cdot [h] \otimes_{B^{\sigma''}} 1_{h^{-1}}\delta_1 \big) \otimes_{\Lambda^e} a \cdot m \\ 
        &= \big(x \cdot [h] \otimes_{B^{\sigma''}}\sigma(h,h^{-1})^{-1}(1_{h^{-1}} \delta_{h^{-1}})(1_h \delta_h) \big)\otimes_{\Lambda^e} a \cdot m \\ 
        &= \xi(h) (1_{h^{-1}} \delta_{h^{-1}}) \cdot \big(x \otimes_{B^{\sigma''}}1_A \delta_1 \big) \cdot (1_h \delta_h)\otimes_{\Lambda^e} a \cdot m \\ 
        &= \xi(h) \big(x \otimes_{B^{\sigma''}}1_A \delta_1 \big) \otimes_{\Lambda^e} (1_h \delta_h) \cdot (a \cdot m) \cdot (1_h \delta_{h^{-1}}) \\ 
        &= \xi(h) \big(x \otimes_{B^{\sigma''}}1_A \delta_1 \big) \otimes_{\Lambda^e} (1_h \delta_h)(a \delta_1) \cdot m \cdot (1_h \delta_{h^{-1}}) \\ 
        &= \xi(h) \big(x \otimes_{B^{\sigma''}}1_A \delta_1 \big) \otimes_{\Lambda^e} (\theta_h(1_{h^{-1}}a)\delta_h) \cdot m \cdot (1_h \delta_{h^{-1}}) \\
        &= \xi(h) \big(x \otimes_{B^{\sigma''}}1_A \delta_1 \big) \otimes_{\Lambda^e} (\theta_h(1_{h^{-1}}a) \delta_1) (1_h \delta_h) \cdot m \cdot (1_h \delta_{h^{-1}}) \\
        &= \big(x \otimes_{B^{\sigma''}}1_A \delta_1 \big) \otimes_{\Lambda^e} (\theta_h(1_{h^{-1}}a) \delta_1) (\xi(h) [h]^{\sigma'} \cdot m) \\
        &= \big(x \otimes_{B^{\sigma''}}1_A \delta_1 \big) \otimes_{\Lambda^e} (([h] \cdot a)\delta_1) \cdot ([h] \cdot m)  \\ 
        &= \tilde{\gamma}_{(X,M)}(x , ([h] \cdot a)\otimes_{A^e}([h] \cdot m)) \\
        &= \tilde{\gamma}_{(X,M)}(x ,[h] \cdot (a\otimes_{A^e}m)).
    \end{align*}
    Then, the following map is well-defined
    \begin{align*}
        \gamma_{(X,M)} : X \otimes_{\kappa_{par}^{\sigma''}G} (A\otimes_{A^e}M) & \to (X \otimes_{B^{\sigma''}} \Lambda)  \otimes_{\Lambda^e}M \\ 
        x \otimes_{\kappa_{par}^{\sigma''}G} (a\otimes_{A^e}m) &\mapsto (x \otimes_{B^{\sigma''}} 1_A \delta_1) \otimes_{\Lambda^e} a \cdot m.
    \end{align*}

    To obtain its inverse map we fix an $m \in M$ and define the map
    \begin{align*}
        \tilde{\psi}_{X,m}: X \times \Lambda & \to  X \otimes_{\kappa_{par}^{\sigma''}G} (A\otimes_{A^e}M) \\ 
        (x, a_h \delta_h)&\mapsto x \otimes_{\kappa_{par}^{\sigma''}G}(1_A \otimes_{A^e} (a_h \delta_h) \cdot m).
    \end{align*}
    Then, $\tilde{\psi}_{X,m}$ is $B^{\sigma''}$-balanced. Indeed,
    \begin{align*}
        \tilde{\psi}_{X,m}(x \cdot e_g^{\sigma''}, a_h \delta_h)    &=  x \cdot e_g^{\sigma''} \otimes_{\kappa_{par}^{\sigma''}G}(1_A \otimes_{A^e} (a_h \delta_h) \cdot m) \\
        &=  x \otimes_{\kappa_{par}^{\sigma''}G} e_g^{\sigma''}  \cdot (1_A \otimes_{A^e} (a_h \delta_h) \cdot m) \\
        (\text{by Lemma \ref{l_teneq} } (iv))&=  x \otimes_{\kappa_{par}^{\sigma''}G} \big( 1_A \otimes_{A^e} e_g \cdot( (a_h \delta_h) \cdot m) \big)  \\
		(\text{by Lemma \ref{l_teneq} } (ii))&=  x \otimes_{\kappa_{par}^{\sigma''}G} \big( 1_A \otimes_{A^e} 1_g ( (a_h \delta_h) \cdot m) 1_g \big)  \\ 
        &=  x \otimes_{\kappa_{par}^{\sigma''}G} \big( 1_g 1_A \otimes_{A^e}  (1_g a_h \delta_h) \cdot m \big) \\
		&=  x \otimes_{\kappa_{par}^{\sigma''}G} \big( 1_A  \otimes_{A^e}  (1_g a_h \delta_h) \cdot m \big)  \\
        &=\tilde{\psi}_{X,m}(x, 1_ga_h \delta_h) \\
        &=\tilde{\psi}_{X,m}(x, e_g^{\sigma''} \cdot (a_h \delta_h)).
    \end{align*}
    Then, the following  map is well-defined
    \begin{align*}
        \psi_{X,m}: X \otimes_{B^{\sigma''}} \Lambda & \to  X \otimes_{\kappa_{par}^{\sigma''}G} (A\otimes_{A^e}M)\\ 
        x \otimes_{B^{\sigma''}} (a_h \delta_h) &\mapsto x \otimes_{\kappa_{par}^{\sigma''}G}(1_A \otimes_{A^e} (a_h \delta_h)\cdot m).
    \end{align*}
    Hence, we can define the map
    \begin{align*}
        \tilde{\psi}_{(X,M)}: (X \otimes_{B^{\sigma''}} \Lambda)  \times M  & \to  X \otimes_{\kappa_{par}^{\sigma''}G} (A\otimes_{A^e}M)\\ 
        (x \otimes_{B^{\sigma''}} a_h \delta_h, m) &\mapsto \psi_{X,m}(x \otimes_{B^{\sigma''}} a_h \delta_h)=x \otimes_{\kappa_{par}^{\sigma''}G}(1_A \otimes_{A^e} (a_h \delta_h) \cdot m),
    \end{align*}
    which is $\Lambda^e$-balanced. Indeed,
    \begin{align*}
        \tilde{\psi}_{(X,M)}((x \otimes_{B^{\sigma''}} (a_h \delta_h))\cdot (b_k \delta_k), m) 
        &=  \tilde{\psi}_{(X,M)}(x \otimes_{B^{\sigma''}} (a_h \delta_h)(b_k \delta_k), m) \\
        &= x \otimes_{\kappa_{par}^{\sigma''}G}(1_A \otimes_{A^e} (a_h \delta_h)(b_k \delta_k) \cdot m) \\ 
        &= x \otimes_{\kappa_{par}^{\sigma''}G}(1_A \otimes_{A^e} (a_h \delta_h)\cdot ((b_k \delta_k) \cdot m)) \\
        &=\tilde{\psi}_{(X,M)}((x \otimes_{B^{\sigma''}} a_h \delta_h),  (b_k \delta_k )\cdot m).
    \end{align*}
    Recall that by the surjective morphism of algebras (\ref{eq:kparGepi}) we are considering $\kappa_{par}^{\sigma''}G$-modules as $\kappa_{par}G$-modules. Thus, for the left action, we have
    \begin{align*}
        \tilde{\psi}_{(X,M)}&((b_k \delta_k) \cdot (x \otimes_{B^{\sigma''}} (a_h \delta_h)), m)  \\ 
        &= \tilde{\psi}_{(X,M)}(x \cdot [k^{-1}] \otimes_{B^{\sigma''}} (b_k \delta_k)(a_h \delta_h), m) \\
        &= x \cdot [k^{-1}] \otimes_{\kappa_{par}^{\sigma''}G} \bigl( 1_A \otimes_{A^e} (b_k \delta_k)(a_h \delta_h) \cdot m \bigr) \\ 
        &= x \cdot [k^{-1}] \otimes_{\kappa_{par}^{\sigma''}G} \bigl( b_k \otimes_{A^e} (1_k \delta_k)(a_h \delta_h) \cdot m \bigr) \\ 
        &= x \cdot [k^{-1}] \otimes_{\kappa_{par}^{\sigma''}G} \bigl( 1_A \otimes_{A^e} (1_k \delta_k)(a_h \delta_h) \cdot m \cdot (b_k \delta_1)\bigr) \\ 
        &= x  \otimes_{\kappa_{par}^{\sigma''}G} [k^{-1}] \cdot \bigl( 1_A \otimes_{A^e} (1_k \delta_k)(a_h \delta_h) \cdot m \cdot (b_k \delta_1)\bigr) \\ 
        &= x  \otimes_{\kappa_{par}^{\sigma''}G} \bigl( [k^{-1}] \cdot  1_A \otimes_{A^e} [k^{-1}] \cdot ((1_k \delta_k)(a_h \delta_h) \cdot m \cdot (b_k \delta_1))\bigr) \\ 
        &\overset{\eqref{e_sigmaprimeisomorphism}}{=} x  \otimes_{\kappa_{par}^{\sigma''}G} \bigl( [k^{-1}] \cdot  1_A \otimes_{A^e} \xi(k)[k^{-1}]^{\sigma'} \cdot ((1_k \delta_k)(a_h \delta_h) \cdot m \cdot (b_k \delta_1))\bigr) \\
        &= x  \otimes_{\kappa_{par}^{\sigma''}G} \bigl( [k^{-1}] \cdot  1_A \otimes_{A^e} (\xi(k)(1_{k^{-1}}\delta_{k^{-1}})(1_k \delta_k)(a_h \delta_h) \cdot m \cdot (b_k \delta_1)(1_k \delta_k))\bigr) \\
        &\overset{\eqref{e_xiformula}}{=}x  \otimes_{\kappa_{par}^{\sigma''}G} \Bigl( 1_{k^{-1}} \otimes_{A^e} (1_{k^{-1}}a_h \delta_h) \cdot m \cdot (b_k \delta_1)(1_k \delta_k)\Bigr) \\ 
        &= x \otimes_{\kappa_{par}^{\sigma''}G} \Bigl(1_{k^{-1}} \otimes_{A^e} \big( (a_h \delta_h) \cdot m \cdot (b_k \delta_1)(1_k \delta_k)\big)\Bigr) \\ 
        &= x  \otimes_{\kappa_{par}^{\sigma''}G} \Bigl( 1_A \otimes_{A^e} \big( (a_h \delta_h) \cdot m \cdot (b_k \delta_1)(1_k \delta_k)(1_{k^{-1}} \delta_1)\big)\Bigr) \\ 
        &= x  \otimes_{\kappa_{par}^{\sigma''}G} \Bigl( 1_A \otimes_{A^e} \big( (a_h \delta_h) \cdot m \cdot (b_k \delta_1)(1_k \delta_k)\big)\Bigr) \\ 
        &= x  \otimes_{\kappa_{par}^{\sigma''}G} \Bigl( 1_A \otimes_{A^e} \big( (a_h \delta_h) \cdot m \cdot (b_k \delta_k)\big)\Bigr) \\
        &= \tilde{\psi}_{(X,M)}((x \otimes_{B^{\sigma''}} a_h \delta_h), m \cdot (b_k \delta_k) ).
    \end{align*}
    Thus, the map
    \begin{align*}
        \psi_{(X,M)} :  (X \otimes_{B^{\sigma''}} \Lambda)  \otimes_{\Lambda^e} M  & \to  X \otimes_{\kappa_{par}^{\sigma''}G} (A\otimes_{A^e}M) \\ 
        (x \otimes_{B^{\sigma''}} a_h \delta_h) \otimes_{\Lambda^e}  m &\mapsto x \otimes_{\kappa_{par}^{\sigma''}G}(1_A \otimes_{A^e} (a_h \delta_h) \cdot m)
    \end{align*}
    is well-defined. Observe that $\gamma_X$ and $\psi_X$ are mutual inverses since
    \begin{align*}
        \psi_{(X,M)}\big( \gamma_{(X,M)}(x \otimes_{\kappa_{par}^{\sigma''}G}(a\otimes_{A^e}m))\big) &=  \psi_{(X,M)}\big( (x \otimes_{B^{\sigma''}} 1_A \delta_1) \otimes_{\Lambda^e} a \cdot m\big) \\ 
        &=x \otimes_{\kappa_{par}^{\sigma''}G}(1_A \otimes_{A^e} a \cdot m) \\
        &=x \otimes_{\kappa_{par}^{\sigma''}G}(a \otimes_{A^e} m)
    \end{align*}
    and
    \begin{align*}
        \gamma_{(X,M)}\big( \psi_{(X,M)}((x \otimes_{B^{\sigma''}} a_h \delta_h) \otimes_{\Lambda^e} m) \big)&= \gamma_{(X,M)} \big(x \otimes_{\kappa_{par}^{\sigma''}G}(1_A \otimes_{A^e} a_h \delta_h \cdot m) \big) \\
        &=  (x \otimes_{B^{\sigma''}} 1_A \delta_1) \otimes_{\Lambda^e} a_h \delta_h  \cdot m \\
        &=(x \otimes_{B^{\sigma''}} a_h \delta_h) \otimes_{\Lambda^e} m.
    \end{align*}

    Let $X'$ be another right $\kappa_{par}^{\sigma''}G$-module and $M'$ be a $\Lambda$-bimodule, $f: X \to X'$ a map of right $\kappa_{par}^{\sigma''}G$-modules and $v: M \to M'$ a map of $\Lambda$-bimodules. Then, the following diagram commutes
    \[\begin{tikzcd}
        {X \otimes_{\kappa_{par}^{\sigma''}G} (A\otimes_{A^e}M)} & {X' \otimes_{\kappa_{par}^{\sigma''}G} (A\otimes_{A^e}M')} \\
        {(X \otimes_{B^{\sigma''}} \Lambda)  \otimes_{\Lambda^e}M} & {(X' \otimes_{B^{\sigma''}} \Lambda)  \otimes_{\Lambda^e}M'}
        \arrow["{\overline{(f,v)}}", from=1-1, to=1-2]
        \arrow["{\underline{(f,v)}}", from=2-1, to=2-2]
        \arrow["{\gamma_{(X,M)}}"', from=1-1, to=2-1]
        \arrow["{\gamma_{(X',M')}}"', from=1-2, to=2-2]
    \end{tikzcd}\]
    where, $\overline{(f,v)}=f \otimes_{\kappa_{par}^{\sigma''}G}(1_A \otimes_{A^e} v)$ and $\underline{(f,v)}=(f \otimes_{B^{\sigma''}}1_A\delta_1)\otimes_{\Lambda^e}v$. Indeed,

    \begin{align*}
        \gamma_{(X',M')} \left( \overline{(f,v)}(x \otimes_{\kappa_{par}^{\sigma''}G}(a \otimes_{A^e} m)) \right)
        &=\gamma_{(X,M)} \left(f(x) \otimes_{\kappa_{par}^{\sigma''}G}(a \otimes_{A^e} v(m))\right) \\
        &= (f(x) \otimes_{B^{\sigma''}} 1_A \delta_1) \otimes_{\Lambda^e} a \cdot v(m) \\
        &= (f(x) \otimes_{B^{\sigma''}} 1_A \delta_1) \otimes_{\Lambda^e} v(a \cdot m) \\ 
        &= \underline{(f,v)}\left( (x \otimes_{B^{\sigma''}} 1_A \delta_1) \otimes_{\Lambda^e} a \cdot m \right) \\
        &= \underline{(f,v)}\left(\gamma_{(X,M)}(x \otimes_{\kappa_{par}^{\sigma''}G}(a \otimes_{A^e} m))  \right).
    \end{align*}
Thus, $\gamma$ is a natural isomorphism.
\end{proof}

\begin{lemma} \label{l_AsmashHisoBAsmashH}
    The map of $\kappa$-modules
    \begin{align*}
        \phi:\Lambda        &\to B^{\sigma} \otimes_{B^{\sigma''}}\Lambda \\ 
            a_h \delta_h    &\mapsto 1_{B^{\sigma}} \otimes_{B^{\sigma''}} a_h \delta_h
    \end{align*}
    is an isomorphism of $\Lambda$-bimodules.
\end{lemma}
\begin{proof}
    First, using \eqref{e_B_module_structure_of_AskewG} it is easy to show that 
    \begin{align*}
        \phi^{-1}: B^{\sigma} \otimes_{B^{\sigma''}}\Lambda  &\to \Lambda \\ 
        w \otimes_{B^{\sigma''}} a_h \delta_h   &\mapsto w\cdot (a_h \delta_h )
    \end{align*}
    is well-defined, and it is the inverse of $\phi$. Thus, $\phi$ an isomorphism of $\kappa$-modules. Now to verify that $\phi$ is a morphism of $\Lambda$-bimodules we choose arbitrary elements $b_k \delta_k, a_h \delta_h$ and $c_t \delta_t$ in $\Lambda$, then
    \begin{align*}
        \phi((b_k \delta_k)(a_h \delta_h)(c_t \delta_t))  
        &= 1_{B^{\sigma}} \otimes_{B^{\sigma''}} (b_k \delta_k)(a_h \delta_h)(c_t \delta_t) \\  
        %&= \big(1_{B^{\sigma}} \otimes_{B^{\sigma''}} (b_k \delta_k)(a_h \delta_h)\big) \cdot (c_t \delta_t) \\ 
        &{ = 1_{B^{\sigma}} \otimes_{B^{\sigma''}} (1_k b_k \delta_k)(a_h \delta_h)  (c_t \delta_t) } \\ 
        &\overset{\eqref{e_B_module_structure_of_AskewG}}{=} 1_{B^{\sigma}} \otimes_{B^{\sigma''}} e_{k}^{\sigma''} \cdot ((b_k \delta_k)(a_h \delta_h) (c_t \delta_t) ) \\ 
        &= 1_{B^{\sigma}} \cdot e_{k}^{\sigma''}\otimes_{B^{\sigma''}} (b_k \delta_k)(a_h \delta_h)  (c_t \delta_t) \\ 
        &\overset{\eqref{e_BsigmaLeftKparGModuleConsequence}}{=} { e^\sigma_{k} } \otimes_{B^{\sigma''}} (b_k \delta_k)(a_h \delta_h)  (c_t \delta_t) \\ 
        &\overset{\eqref{e_xisigma}}{=} \big( \xi(k) [k]^\sigma 1_{B^{\sigma}}[k^{-1}]^{\sigma}\otimes_{B^{\sigma''}} (b_k \delta_k)(a_h \delta_h)\big) \cdot (c_t \delta_t) \\ 
        &= \big(1_{B^{\sigma}} \cdot \xi(k)[k^{-1}]^{\sigma'}\otimes_{B^{\sigma''}} (b_k \delta_k)(a_h \delta_h)\big) \cdot (c_t \delta_t) \\
        &\overset{\eqref{e_sigmaprimeisomorphism}}{=} \big(1_{B^{\sigma}} \cdot [k^{-1}]^{\sigma''}\otimes_{B^{\sigma''}} (b_k \delta_k)(a_h \delta_h)\big) \cdot (c_t \delta_t) \\
        &= (b_k \delta_k) \cdot \big(1_{B^{\sigma}} \otimes_{B^{\sigma''}} (a_h \delta_h)\big) \cdot (c_t \delta_t) \\
        &= (b_k \delta_k) \cdot \phi (a_h \delta_h) \cdot (c_t \delta_t).
    \end{align*}
    Thus, $\phi$ is an isomorphism of $\Lambda$-bimodules.
\end{proof}

\begin{corollary} \label{c_F2F1F}
    The functors $F_2F_1$ and $F$ are naturally isomorphic. 
\end{corollary}

\begin{proof}
    Using Proposition \ref{p_inXB} for the particular case $X=B^{\sigma}$ we obtain that the functors
    \[
        B^{\sigma}\otimes_{\kappa_{par}^{\sigma''}G} (A \otimes_{A^e}-): \Lambda^e\textbf{-Mod}  \to \kappa\textbf{-Mod}
    \]
    and
    \[
        (B^{\sigma} \otimes_{B^{\sigma''}} \Lambda) \otimes_{\Lambda^e}-:\Lambda^e\textbf{-Mod}  \to \kappa\textbf{-Mod}
    \] 
    are naturally isomorphic. Clearly, $F_2F_1 = B^{\sigma''}\otimes_{\kappa_{par}^{\sigma''}G} (A \otimes_{A^e}-)$. On the other hand, by Lemma \ref{l_AsmashHisoBAsmashH} we know that $B^{\sigma} \otimes_{B^{\sigma''}} \Lambda \cong \Lambda$ as $\Lambda$-bimodules,  
    which implies the natural isomorphism  $F \cong (B^{\sigma} \otimes_{B^{\sigma''}} \Lambda) \otimes_{\Lambda^e}- ,$ and  the desired conclusion follows.
\end{proof}

Since our objective is to construct a homological Grothendieck spectral sequence (see for example \cite[Theorem 10.48]{RotmanAnInToHoAl}) we need to show that the functor $F_1$ sends projective $\Lambda^e$-modules to left $F_2$-acyclic $\kappa_{par}^{\sigma''}G$-modules. For this purpose, we need the following results:

\begin{lemma} \label{l_fgIdmComIsVNR}
    Any commutative algebra finitely generated by idempotents is Von Neumann regular.
\end{lemma}

\begin{proof}
    Let $\mathcal{A}$ be a commutative algebra finitely generated by idempotents. Denote by  $E$  the semigroup of the idempotent elements of $\mathcal{A}$ and by $\kappa E$ as the semigroup algebra of $E$ over $\kappa .$ Then the map $\varphi: \kappa E \to A$ given by $\varphi(e):= e$, for all $e \in E$, is a surjective morphism of algebras. Obviously, $E$ is a meet semilattice, and hence $\kappa E$ is the Möbius algebra of $E$ as defined in  \cite{SOLOMON1967603}. By \cite[Theorem 1]{SOLOMON1967603} there exists an orthogonal basis of idempotents $\{v_i\}_{i=0}^n$ of $\kappa E$, such that $\sum_{i=0}^n v_i = 1_{\kappa E}$. Therefore, $\{\varphi(v_i)\}_{i=0}^n$ is a generator set of $\mathcal{A}$ of orthogonal idempotents and, consequently, $V_{\mathcal{A}} = \{\varphi(v_i)\}_{i=0}^n \setminus \{ 0 \}$ is a basis of orthogonal idempotents of $\mathcal{A}$ such that $\sum_{x \in V_{\mathcal{A}}} x = 1_{\mathcal{A}}$. Thus, by \cite[Lemma 1.6]{goodearl1979neumann} we conclude that $\mathcal{A}$ is Von Neumann regular.
\end{proof}

\begin{proposition} \label{p_ComIdAlgIsVNR}
    Any commutative algebra generated by idempotents is Von Neumann regular.
\end{proposition}

\begin{proof}
    Let $\mathcal{A}$ be a commutative algebra generated by idempotents. For any $x \in \mathcal{A}$ we have that
    \[
      x = \sum_{i = 0}^n \alpha_i e_i,  
    \]
    where the $e_i$ are idempotents and $\alpha_i \in \kappa$. Let $\mathcal{E}$ be the subalgebra of $\mathcal{A}$ generated by $\{e_i\}_{i=0}^n$. Then, by Lemma \ref{l_fgIdmComIsVNR} we have that $\mathcal{E}$ is Von Neumann regular, and thus there exist $y \in \mathcal{E} \subseteq \mathcal{A}$ such that $xyx = x$. Henceforth, $\mathcal{A}$ is Von Neumann regular. 
\end{proof}

Since all modules over a Von Neumann regular ring are flat   \cite[Corollary 1.13]{goodearl1979neumann}, a direct  consequence of Proposition~\ref{p_ComIdAlgIsVNR} is the following:

\begin{corollary} \label{l_Bflat}
    The algebra $B^{\sigma''}$ is Von Neumann regular. Consequently, $- \otimes_{B^{\sigma''}} \Lambda$ is an exact functor.
\end{corollary}

Now we have the necessary tools to show the following theorem.

\begin{theorem} \label{t_SSHHTor}
    There exists a first quadrant homological spectral sequence
    \[
        E^2_{p,q} = \operatorname{Tor}^{\kappa_{par}^{\sigma''}G}_p(B^\sigma,H_q(A,M) ) \Rightarrow H_{p+q}(\Lambda, M). 
    \]
\end{theorem}

\begin{proof}
    Notice that $L_\bullet F_1(-)=H_\bullet(A,-) $, $L_\bullet F_2 (-) = \operatorname{Tor}^{\kappa_{par}^{\sigma''}G}_\bullet (B^\sigma,-)$ and $L_{\bullet}F(-)=H_{\bullet}(\Lambda, -)$. We know that $F_1$ and $F_2$ are right exact functors and by Corollary \ref{c_F2F1F} we have that $F_2 F_1 \cong F$. So, it only remains to verify that for any projective object $P$ of $\Lambda^e$\textbf{-Mod} we have that $F_1(P)$ is left $F_2$-acyclic. That is, we want to see that
    \[
        \operatorname{Tor}^{\kappa_{par}^{\sigma''}G}_n(B^\sigma,F_1(P))=0, \, \forall n>0.
    \]
    Since,
    \begin{align*}
        \operatorname{Tor}^{\kappa_{par}^{\sigma''}G}_n(B^\sigma,F_1(P))
        &=\operatorname{L}_n(- \otimes_{\kappa_{par}^{\sigma''}G} (A \otimes_{A^e} P))(B^\sigma) \\ 
       (\text{by Proposition } \ref{p_inXB}) &\cong  \operatorname{L}_n((- \otimes_{B^{\sigma''}} \Lambda) \otimes_{\Lambda^e}P)(B^\sigma),
    \end{align*}
    by Lemma \ref{l_Bflat} we have that the functor  $- \otimes_B \Lambda$ is exact, and since $P$ is projective, then $-\otimes_{\Lambda^e}P$ is also exact. Therefore, $(- \otimes_B \Lambda) \otimes_{\Lambda^e}P$ is an exact functor, whence we obtain
    \[
        \operatorname{Tor}^{\kappa_{par}^{\sigma''}G}_n(B^\sigma, F_1(P))=0, \, \forall n>0. 
    \]
\end{proof}

Now the objective is to understand the behavior of the functor $\operatorname{Tor}^{\kappa_{par}^{\sigma''}G}_n(B^\sigma, - )$, principally we want to relate $\operatorname{Tor}^{\kappa_{par}^{\sigma''}G}_\bullet(B^\sigma, X)$ and the partial homology groups $H_\bullet^{par}(G, X),$ where $X$ is a left $\kappa_{par}^{\sigma''}G$-module. Recall from 
 \cite{MMAMDDHKPartialHomology} the $n$-th partial homology group of $G$ with coefficients in $X$ is defined by
$$ H_n^{par}(G,X)= \operatorname{Tor}^{\kappa_{par} G}_n(B, X ).$$

  First, we need the following proposition

\begin{proposition} \label{p_FlatSR2}
    Let $W$ be a subset of idempotents of $B$ and $\mathcal{W}$ the two-sided ideal of $\kappa_{par}G$ generated by $W$, then $\kappa_{par}G / \mathcal{W}$ is a flat right (left) $\kappa_{par}G$-module.
\end{proposition}

\begin{proof}
   Recall that each element of the semigroup $S(G)$ can be written in the form $u[g]$ for some $g\in G$ and some idempotent $u \in S(G)$ (see \cite{exe}).   Let $I$ be a left ideal of $\kappa_{par}G$ and $z \in \mathcal{W} \cap I.$ Then there exist $\lambda_i \in \kappa$, $w_i \in W$, \,  $[g_i], \, x_i \in \kappa_{par}G$ and idempotents  $u_i \in S(G)$  such that
    \[
        z= \sum_{i=1}^{n} \lambda_i u_i [g_i]w_ix_i.
    \]
    Recall that $[g_i]w_i[g^{-1}_i] \in  B$ is an idempotent and, obviously, $[g_i]w_i[g^{-1}_i] \in \mathcal{W}$. Then,
    \[
        z= \sum_{i=1}^{n} \lambda_i u_i [g_i]w_i [g_i^{-1}][g_i]x_i =\sum_{i=1}^{n} \lambda_i ([g_i]w_i[g_i^{-1}]) u_i[g_i]x_i.
    \]
    Let $\mathcal{A}$ be the unital $K$-algebra generated by $\{[g_i]w_i[g_i^{-1}] : 1 \leq i \leq n \} \cup \{ 1 \}$, then by Proposition \ref{p_ComIdAlgIsVNR} we have that $\mathcal{A}$ is a Von Neumann regular algebra. Let $\mathcal{J}$ be the ideal of $\mathcal{A}$ generated by $\{[g_i]w_i[g_i^{-1}] : 1 \leq i \leq n \}$. Since $\mathcal{J}$ is an ideal in a commutative algebra generated by finitely many idempotents, $\mathcal{J}$ has a unit element $w_0$.
    Therefore,
    \begin{align*}
        z=\sum_{i=1}^{n} \lambda_i ([g_i]w_i[g_i^{-1}]) u_i[g_i]x_i 
        &=\sum_{i=1}^{n} \lambda_i w_0([g_i]w_i[g_i^{-1}]) u_i[g_i]x_i \\
        &= w_0\sum_{i=1}^{n} \lambda_i ([g_i]w_i[g_i^{-1}]) u_i[g_i]x_i \\
        &=w_0 z. 
    \end{align*}
    Observe that since $\mathcal{A} \subseteq \mathcal{W}$, we have that $w_0 \in \mathcal{W}$. Hence, $z=w_0 z \in \mathcal{W} \cdot I$. Therefore, $\mathcal{W} \cap I \subseteq \mathcal{W} \cdot I$. Clearly, $\mathcal{W} \cdot I \subseteq \mathcal{W} \cap I$. Thus, \cite[Proposition 4.14]{Lam1998LecturesOM} implies that $\kappa_{par}G / \mathcal{W}$ is flat as right $\kappa_{par}G$-module. The proof of its flatness as a left $\kappa_{par}G$-module is symmetric. 
\end{proof}

\begin{corollary} \label{c_FlatSR2}
    Let $W$ be a subset of $B$ and $\mathcal{W}$ the two-sided ideal of $\kappa_{par}G$ generated by $W.$ Then, $\kappa_{par}G / \mathcal{W}$ is a flat right (left) $\kappa_{par}G$-module.
\end{corollary}

\begin{proof}
    Let $w \in \mathcal{W}$, then
    \[
        w = \sum_{i=0}^n \alpha_i e_i,   
    \]
    where $\alpha_i \in \kappa$ and $\{e_i\}_{i=0}^n$ is a set of idempotents in $B$. Define $\mathcal{A} \subseteq B$ as the subalgebra generated by $\{e_i\}_{i=0}^n$. Thus, as in the proof of Lemma \ref{l_fgIdmComIsVNR}, there exist a basis of orthogonal idempotents $\{v_j\}_{j=0}^m$ of $\mathcal{A}$. Thus,
    \[
        w = \sum_{j=0}^m \beta_j v_j,   \text{ where } \beta_j \in \kappa.
    \]
    Therefore, $v_j = \beta_j^{-1}w v_j \in \mathcal{W}$, for all $0 \leq j \leq m$. Then, we conclude that any element of $\mathcal{W}$ is generated by a set of idempotents of $\mathcal{W} \cap B$. Thus, by Proposition \ref{p_FlatSR2} we have that $\kappa_{par}G/\mathcal{W}$ is flat as left (right) $\kappa_{par}G$-module.
\end{proof}

We know by the proof of Theorem 10.9 from \cite{E6} that  $B$ is the (universal) algebra generated by the symbols $e_g, (g\in G),$ subject to the relations
$$e_1 = 1_B, \;\;\; e^2_g=e_g, \;\;\; e_g e_h = e_h e_g, \;\;\; \forall g,h \in G.$$   Therefore, we can   define the homomorphism of algebras $\zeta  : B \to B^\sigma$ by setting $\zeta(e_g) = e_g^\sigma$. Now let $J$  the  two-sided ideal of $\kappa_{par} G$ generated by $\ker \zeta$ and define 
\begin{equation}
   \Omega_\sigma := \kappa_{par}G/ J.
\end{equation}
Notice that if $\sigma(g,h)=1$ for all $g,h \in G$ we have that $B=B^\sigma$, thus $\Omega_\sigma = \kappa_{par}G$. Observe that the following fact is a direct consequence of Corollary \ref{c_FlatSR2}.

\begin{lemma}
    $\Omega_\sigma$ is flat as a left (right) $\kappa_{par}G$-module. 
\end{lemma}

Recall that $B^{\sigma}$ is a left (right) $\kappa_{par}G$-module by \eqref{eq:BleftKparMod}, \eqref{e_sigmaprimeisomorphism} and \eqref{eq:kparGepi}. Explicitly, 
\[
    [g] \cdot w = [g]^{\sigma''} \cdot w = \xi(g) [g]^{\sigma'}\cdot w = \xi(g) [g]^\sigma w [g^{-1}]^\sigma, \, \forall g \in G, w \in B^\sigma.
\]
Furthermore, $B^\sigma$ is a left (right) $\Omega_\sigma$-module. Indeed, we only have to verify that $z \cdot w = 0$ for all $z \in \ker \zeta$. First observe that due \eqref{e_BsigmaLeftKparGModuleConsequence} we have for all $w \in B^\sigma$ that
    \[
        e_g \cdot w = e^{\sigma''}_g \cdot w  = e_g^\sigma w.
    \]
Then, $v \cdot w = \zeta(v)w$ for all $v \in B$ and $w \in B^\sigma$, hence $ v \cdot w = 0 \text{ for all } v \in \ker \zeta.$ Consequently, $B^\sigma$ is a left $\Omega_\sigma$-module. The fact that it is a right $\Omega_\sigma$-module is symmetric.

Next, observe that $\zeta$ is also a homomorphism of right $\kappa_{par}G$-modules. Indeed, using \eqref{Commuting e} for the canonical partial projective representation we see that
  \begin{align*}
  \zeta ((e_{h_1} \ldots e_{h_n})\cdot [g])
  &= \zeta ([g\m](e_{h_1} \ldots e_{h_n}) [g]) \\
  &= \zeta (e_{g\m h_1} \ldots e_{g\m h_n} e_{g\m}) \\
  &=e^\sigma_{g\m h_1} \ldots e^\sigma_{g\m h_n} e^\sigma_{g\m} \\
  &\overset{\eqref{e_xisigma}}{=}e^\sigma_{g\m h_1} \ldots e^\sigma_{g\m h_n} \xi(g)[g^{-1}]^\sigma [g]^\sigma \\
  &\overset{\eqref{Commuting e}}{=}\xi(g)[g^{-1}]^\sigma e^\sigma_{h_1} \ldots e^\sigma_{h_n}  [g]^\sigma \\
  &=(e^\sigma_{h_1} \ldots e^\sigma_{ h_n}) \cdot  (\xi(g)[g]^{\sigma'}) \\ 
  &\overset{\eqref{e_sigmaprimeisomorphism}}{=}\zeta (e_{h_1} \ldots e_{h_n}) \cdot  [g]^{\sigma''} \\ 
  &=\zeta (e_{h_1} \ldots e_{h_n}) \cdot  [g],
  \end{align*} 
as desired. It follows that  $\ker \zeta $ is a submodule of the right $\kappa_{par}G$-module $B.$

\begin{lemma}\label{B_Sigma_iso}
    $B \otimes_{\kappa_{par}G} \Omega_\sigma \cong B^\sigma$ as right $\kappa_{par}G$-modules.
\end{lemma}

\begin{proof}
    We have the following exact sequence of right $\kappa_{par}G$-modules
    \[
        0 \to \ker \zeta \to B \to B^\sigma \to 0,
    \]
    and, since $\Omega_\sigma$ is flat, then the following sequence is exact
    \[
        0 \to \ker \zeta \otimes_{\kappa_{par}G} \Omega_\sigma \to B \otimes_{\kappa_{par}G} \Omega_\sigma \to B^\sigma \otimes_{\kappa_{par}G} \Omega_\sigma \to 0.
    \]
    Since $B$ is Von Neumann regular, then $(\ker \zeta)^2 = \ker \zeta$, thus $\ker \zeta \otimes_{\kappa_{par}G} \Omega_\sigma = 0$. Hence,
    \[
        B \otimes_{\kappa_{par}G} \Omega_\sigma \cong B^\sigma \otimes_{\kappa_{par}G} \Omega_\sigma \cong B^\sigma \otimes_{\Omega_\sigma} \Omega_\sigma \cong B^\sigma.
    \]
\end{proof}

\begin{lemma} \label{l_OmegaSigmaIsKparSigmaPrimeModule}
        $\Omega_\sigma$ is a left (right) $\kappa_{par}^{\sigma''}G$-module, via
        \[
            [g]^{\sigma''} \cdot x = [g] \cdot x = \overline{[g]} x \text{ and } x \cdot [g]^{\sigma''} = x \cdot [g] =x\overline{[g]}.
        \]
\end{lemma}

\begin{proof}
    Define $\Gamma: G \to \Omega_\sigma$, such that $\Gamma_g:= \overline{[g]}$, since $\sigma''$ is a trivial factor set and $\Omega_\sigma$ is a quotient algebra of $\kappa_{par}G$ we only have to verify the first relation of Definition \ref{def:parproj}. Observe that
    \begin{align*}
        \text{ if }\sigma''(g,h)=0 
        &\Rightarrow \sigma'(g,h)=0 \Rightarrow \sigma(g,h)=0 \\ 
        & \Rightarrow e_g^\sigma e_{gh}^\sigma= e_{g^{-1}}^\sigma e_{h^{-1}g^{-1}}^\sigma = 0 
        \Rightarrow e_g e_{gh}, \, e_{h^{-1}g^{-1}}e_{h^{-1}} \in \ker \zeta \\
        &\Rightarrow [g^{-1}][gh] = [g^{-1}]e_g e_{gh} [gh] \in J \text{ and } [gh][h^{-1}] = [gh]e_{h^{-1}g^{-1}}e_{h^{-1}}[h^{-1}]  \in J \\
        &\Rightarrow \Gamma_{g^{-1}}\Gamma_{gh} = \overline{[g^{-1}][gh]} = 0 \text{ and } \Gamma_{gh} \Gamma_{h^{-1}} = \overline{[gh][h^{-1}]}=0.
    \end{align*}
    Therefore, by the universal property of $\kappa_{par}^{\sigma''}G$, there exists a morphism of algebras $\kappa_{par}^{\sigma''}G \to \Omega_\sigma$ such that $[g]^{\sigma''} \mapsto \overline{[g]}$. Furthermore, notice that this morphism is surjective. Finally, the $\kappa_{par}^{\sigma''}G$-module structure of $\Omega_\sigma$ is determined by this morphism.
\end{proof}

Notice that by Lemma \ref{l_OmegaSigmaIsKparSigmaPrimeModule} we have that the $\kappa_{par}G$-module structure of $\Omega_\sigma$ is that induced by the surjective morphism of algebras $\kappa_{par}G \to \kappa_{par}^{\sigma''}G$ and the $\kappa_{par}^{\sigma''}G$-module structure of $\Omega_\sigma$, thus $\Omega_\sigma \otimes_{\kappa_{par}^{\sigma''}G}X \cong \Omega_\sigma \otimes_{\kappa_{par}G}X$ for any $\kappa_{par}^{\sigma''}G$-module $X$.

\begin{proposition}\label{prop:TorSigma'Hkpar}
    Let $X$ be a $\kappa_{par}^{\sigma''}G$-module. Then,
    \[
        \operatorname{Tor}_\bullet^{\kappa_{par}^{\sigma''}G}(B^\sigma, X) \cong H^{par}_\bullet (G, \Omega_\sigma \otimes_{\kappa_{par}G}X). 
    \]
\end{proposition}

\begin{proof}
    We know that $\Omega_\sigma$ is flat on either side over $\kappa_{par}G$. Since \eqref{eq:kparGepi} is a surjective map of algebras, then $\Omega_\sigma$ is flat as $\kappa_{par}^{\sigma''}G$-module on either side. Thus, by Lemma~\ref{B_Sigma_iso} and \cite[Corollary 10.61]{RotmanAnInToHoAl} we have that
    \begin{align*}
        \operatorname{Tor}_\bullet^{\kappa_{par}^{\sigma''}G}(B^\sigma, X) 
        &\cong \operatorname{Tor}_\bullet^{\kappa_{par}^{\sigma''}G}(B \otimes_{\kappa_{par}G} \Omega_\sigma, X) \\
        & \cong  \operatorname{Tor}_\bullet^{\kappa_{par}G}(B, \Omega_\sigma \otimes_{\kappa_{par}^{\sigma''}G} X) \\ 
        & \cong  \operatorname{Tor}_\bullet^{\kappa_{par}G}(B, \Omega_\sigma \otimes_{\kappa_{par}G} X) \\ 
        &= H^{par}_\bullet (G, \Omega_\sigma \otimes_{\kappa_{par}G} X).
    \end{align*} 
\end{proof}

We proceed with the following lemma:

\begin{lemma}  \label{l_AoMisOmegaModule}
  The epimorphism \eqref{eq:kparGepi} induces a left  $\Omega_\sigma$-module structure on $A \otimes_{A^e}M.$
\end{lemma}

\begin{proof}
 It is enough  to show that $\ker \zeta $ annihilates $A \otimes_{A^e}M.$ Since $\zeta  : B \to B^\sigma$ can be factored as $B \to B^{\sigma''} \stackrel{\iota}{\to}  B^\sigma,$ with $e_g \mapsto e^{\sigma''}_g \overset{\iota}{\mapsto} e^{\sigma}_g,$ where $B \to B^{\sigma''}$ is the restriction of \eqref{eq:kparGepi} to $B,$ and $\iota$ is the map defined in \ref{r_iota}, it suffices to verify that $\ker \iota$ annihilates  $A \otimes_{A^e} M.$

    Observe that by Lemma \ref{l_teneq} $(iv)$ and the equation \eqref{e_KparG-modM} we have
    \begin{align*}
        e_g^{\sigma''} \cdot (a \otimes_{A^e} m) 
        &= a \otimes_{A^e} e_g^{\sigma''} \cdot m \\ 
        &= a \otimes_{A^e} [g]^{\sigma''}[g^{-1}]^{\sigma''} \cdot m \\ 
        &= a \otimes_{A^e} \xi(g) \xi(g^{-1})[g]^{\sigma'}[g^{-1}]^{\sigma'} \cdot m \\ 
        &= a \otimes_{A^e} \left(\xi(g)\xi(g^{-1})(1_g \delta_g)(1_{g^{-1}} \delta_{g^{-1}}) \cdot m \cdot (1_g \delta_g)(1_{g^{-1}} \delta_{g^{-1}})\right) \\ 
        &\overset{\eqref{e_xiformula}}{=} a \otimes_{A^e} \left(\xi(g)\xi(g)(1_g \delta_g)(1_{g^{-1}} \delta_{g^{-1}}) \cdot m \cdot (1_g \delta_g)(1_{g^{-1}} \delta_{g^{-1}})\right) \\ 
        &= a \otimes_{A^e} \left((1_g \delta_1) \cdot m \cdot (1_g \delta_1)\right) \\
        & \overset{\eqref{e_B_module_structure_of_AskewG}}{=} a \otimes_{A^e} \left((e_g^\sigma \cdot (1 \delta_1)) \cdot m \cdot (e_g^\sigma \cdot (1 \delta_1))\right)
    \end{align*}
    whence we obtain
    \begin{equation} \label{e_iotaTensor}
        w \cdot (a \otimes_{A^e} m) = a \otimes_{A^e} \left((\iota(w) \cdot (1_A\delta_1)) \cdot m \cdot (\iota(w) \cdot (1_A\delta_1))\right), \, \text{ for all } w \in B^{\sigma''}.
    \end{equation}
    Now let $z \in \ker \iota \subseteq B^{\sigma''}$, $a \in A$ and $m \in M$. By \eqref{e_iotaTensor} we have
    \[
        z \cdot (a \otimes_{A^e}m) = a \otimes_{A^e} \left((\iota(z) \cdot (1_A\delta_1)) \cdot m \cdot (\iota(z) \cdot (1_A\delta_1))\right) = 0.
    \]
   Therefore, the left $ \Omega_\sigma$-action, given by  $\overline{x} \cdot (a \otimes_{A^e}m) := x \cdot (a \otimes_{A^e}m),$  for all $\overline{x} \in \Omega_\sigma ,$ is well-defined.
\end{proof}
Analogously to the computations carried out in the proof of Lemma~\ref{l_AoMisOmegaModule}, it can be shown that $z \cdot m = (\iota(z) \cdot (1_A\delta_1)) \cdot m \cdot (\iota(z) \cdot (1_A\delta_1)) $ for all $z \in B^{\sigma''}$ and $m \in M$. As result, $\ker \zeta$ annihilates $M$, consequently, $M$ also has a structure of an  $\Omega_\sigma$-module (a fact which will not be used below).
%Since the $\kappa_{par}^{\sigma''}G$-module structure of $M$ is determined by $\Lambda$, and $\Lambda$ depends on $\sigma$, we can show by direct computations that $M$ is an $\Omega_\sigma$-module.

It is easy to verify that $\Omega_\sigma$\textbf{-Mod} is a full subcategory of $\kappa_{par}^{\sigma''}G$\textbf{-Mod}. Therefore, Proposition \ref{l_AoMisOmegaModule} implies that we can restrict the target of the functor $F_1$ to the subcategory $\Omega_\sigma$\textbf{-Mod}. Hence, the Hochschild homology groups $H_q(A,M)$ are in fact $\Omega_\sigma$-modules via the  epimorphism \eqref{eq:kparGepi} and, consequently,   
$\Omega_\sigma  \otimes_{\kappa_{par}G} H_q(A,M) \cong \Omega_\sigma  \otimes_{\Omega_\sigma} H_q(A,M) \cong  H_q(A,M).$  Thus, thanks to 
Proposition~\ref{prop:TorSigma'Hkpar}, we can reformulate  Theorem \ref{t_SSHHTor} as follows.

\begin{theorem} \label{t_SSHH}
    Let $A$ be a $\kappa$-algebra, $\theta$ a unital partial action of a group $G$ on $A$ with   twist 
    $\sigma \in pm(G)$. Then, for any $A \ast_{\theta,\sigma} G$-bimodule $M$ there exists a first quadrant homological spectral sequence
    \[
        E^2_{p,q} = H_p^{par}(G,H_q(A,M) ) \Rightarrow H_{p+q}(A \ast_{\theta, \sigma} G, M). 
    \]
\end{theorem}

\begin{example}
    If the algebra $A$ is separable, then the spectral sequence collapses, and we obtain the following isomorphism
    \[
        H_n^{par}(G,M/[A,M] ) \cong H_{n}(A \ast _\sigma G, M).
    \]
    
\end{example}

\begin{example}
Let $M$ be a $\kappa_{par}^{\sigma}G$-bimodule. Since, the enveloping algebra $B^e$ is generated by the set of idempotents $\{e_g \otimes e_h: g,h \in G\},$  then, by Proposition \ref{p_ComIdAlgIsVNR}, $B$ is flat as a $B^e$-module, so that the spectral sequence collapses in the $p$-axis. Hence, we obtain the following generalization of the MacLane isomorphism
\[
    H_n^{par}(G, M/ [B,M]) \cong H_n(\kappa_{par}^{\sigma}G, M).
\]
The above isomorphism generalizes the MacLane isomorphism (see for example \cite[7.4.2]{loday2013cyclic}) in the sense that if $M$ is a $G$-module. Then, 
\[
    H_{\bullet}(\kappa_{par}G,M)= H_{\bullet}(\kappa G,M) \text{ and } H_\bullet(B,M) = H_\bullet(\kappa,M).
\]
Indeed, notice that $\kappa G \cong \kappa_{par}G/ < 1 - e_g \mid g \in G>$, thus, by Proposition \ref{p_FlatSR2} we have that $\kappa G$ is flat as left (right) $\kappa_{par}G$-module. Therefore, by Lemma \ref{l_RSprojectiveIsPSeProjective} ${(\kappa G)^e}$ is projective as $(\kappa_{par}G)^e$-module. Then,
\begin{align*}
    H_\bullet(\kappa_{par}G, M) &= \operatorname{Tor}_\bullet^{(\kappa_{par}G)^e}(\kappa_{par}G, M) \\ 
    &= \operatorname{Tor}_\bullet^{(\kappa_{par}G)^e}(\kappa_{par}G,{(\kappa G)^e} \otimes_{(\kappa G)^e} M) \\ 
    &= \operatorname{Tor}_\bullet^{(\kappa G)^e}(\kappa_{par}G \otimes_{(\kappa_{par}G)^e} (\kappa G)^e ,M) \\ 
    &= \operatorname{Tor}_\bullet^{(\kappa G)^e}(\kappa G \otimes_{(\kappa G)^e} (\kappa G)^e ,M) \\
    &= \operatorname{Tor}_\bullet^{(\kappa G)^e}(\kappa G ,M) \\
    &= H_\bullet(\kappa G, M).
\end{align*}
Analogously, one shows that $H_\bullet(B,M) = H_\bullet(\kappa,M)$ since $\kappa$ is flat as left (right) $B$-module (via the morphism $e_g \in B\mapsto 1 \in \kappa$). Since, $H_0(\kappa,M)=M$ and $H_n(\kappa,M)=0$ for $n>0$, we obtain
\[
    H_{\bullet}(\kappa G,M) = H_{\bullet}(\kappa_{par}G,M) \cong H_\bullet^{par}(G, M) =  H_\bullet(G, M).
\]
\end{example}

\section{Cohomological spectral sequence}\label{sec:CohomolSpecSeq}

We will now construct the dual spectral sequence of the Theorem \ref{t_SSHH} for Hochschild's cohomology. Since the computations are similar to those done in the previous section and \cite{AlAlRePartialCohomology} (the difference is that we need to take into account the details about $\sigma$ as in the previous section) we will skip most of them.

As in the previous section, we will fix a twisted unital partial action $\theta$ of $G$ on a $\kappa$-algebra $A$ with twist $\sigma$, such that $\sigma(g,g^{-1})=1$ for all $g \in G \setminus G_2$. If $M$ is a $\Lambda$-bimodule, we can define a projective partial $\sigma'$-representation $\Gamma:G \to \operatorname{End}_\kappa(\hom_{A^e}(A,M))$ by  
\begin{equation}
    \Gamma_g(f)(a)= [g]^{\sigma'} \cdot f( [g^{-1}] \cdot a).
\end{equation}
Thus, by the universal property of $\kappa_{par}^{\sigma'}G$ we have that $\hom_{A^e}(A,M)$ is a $\kappa_{par}^{\sigma'}G$-module. Henceforth, by \eqref{e_sigmaprimeisomorphism} we have that $\hom_{A^e}(A,M)$ is $\kappa_{par}^{\sigma''}G$-module with the action given by 
\begin{equation}
    ([g]^{\sigma''} \cdot f)(a):= \xi(g)([g]^{\sigma'} \cdot f)(a)=\xi(g)[g]^{\sigma'} \cdot f( [g^{-1}] \cdot a).
\end{equation}
Hence, we can define the left exact functors
\begin{equation}
    W_1:= \hom_{A^e}(A,-):\Lambda^e\textbf{-Mod} \to \kappa_{par}^{\sigma''}G\textbf{-Mod},
\end{equation}
and
\begin{equation}
    W_2:=\hom_{\kappa_{par}^{\sigma''}G}(B^\sigma,-): \kappa_{par}^{\sigma''}G\textbf{-Mod} \to \kappa\textbf{-Mod}.
\end{equation}
The functor that determines the Hochschild cohomology of $\Lambda$ with coefficients in $M$ is 
\begin{equation}
    W:= \hom_{\Lambda^e}(\Lambda,-): \Lambda^e \textbf{-Mod} \to \kappa\textbf{-Mod}.
\end{equation}

\begin{remark} \label{r_Hochschild_cohomology}
    Recall that the $A^e$-module structure of $\Lambda^e$ is determined by the morphism of rings $A^e \to \Lambda^e$ induced by the natural inclusion of $A$ into $\Lambda$. Since, by Lemma~\ref{l_AskewP}  the $K$-algebra $\Lambda^e$ is projective as an $A^e$-module, it follows by \cite[Corollary 3.6A]{Lam1998LecturesOM} that any injective $\Lambda^e$-module is an injective $A^e$-module. Consequently,
     \[
         H^\bullet(A,M) \cong R^\bullet W_1(M),
     \]
     for any $\Lambda^e$-module $M$.
 \end{remark}

Let $X$ be a left $\kappa_{par}^{\sigma''}G$-module then by Proposition \ref{p_left_and_right_partial_modules_equivalence} and Proposition \ref{p_ArtG-Mod-XtensorArtG} we have that $X \otimes_{B^{\sigma''}} \Lambda$ is a $\Lambda^e$-module with actions
\begin{equation}
    a_g \delta_g \cdot (x \otimes_{B^{\sigma''}} c_t \delta_t) := [g]^{{\sigma''}} \cdot x \otimes_{ B^{\sigma''}} (a_g \delta_g)(c_t \delta_t)
\end{equation}
and
\begin{equation} 
    (x \otimes_{B^{\sigma''}} c_t \delta_t) \cdot a_g \delta_g := x \otimes_{ B^{\sigma''}} (c_t \delta_t)(a_g \delta_g).
\end{equation}

Now consider the following linear isomorphism
\[
    \gamma: \hom_{\kappa_{par}^{\sigma''}G}(X, \hom_{A^e}(A,M)) \to \hom_{\Lambda^e}(X \otimes_{B^{\sigma''}}\Lambda,M)
\]
such that for $f \in \hom_{\kappa_{par}^{\sigma''}G}(X, \hom_{A^e}(A,M))$
\[
    \gamma(f)(x \otimes_{B^{\sigma''}} a_g \delta_g)= f_x(1_A) \cdot a_g \delta_g,
\]
where $f_x:=f(x)$. The inverse map of $\gamma$ is
\[
    \gamma^{-1}: \hom_{\Lambda^e}(X \otimes_{B^{\sigma''}}\Lambda,M) \to \hom_{\kappa_{par}^{\sigma''}G}(X, \hom_{A^e}(A,M))
\]
such that for $f \in \hom_{\Lambda^e}(X \otimes_{B^{\sigma''}}\Lambda,M)$
\[
    \gamma^{-1}(f)_x(a)= f(x \otimes_{B^{\sigma''}} a \delta_1).
\]
The isomorphism $\gamma$ determines the natural isomorphisms given in the following proposition

\begin{proposition} \, \label{p_cohomology_functors_isomorphism}
    \begin{enumerate}[(i)]
        \item The functors
        \[
            \hom_{\kappa_{par}^{\sigma''}G}(-, \hom_{A^e}(A,M)): \kappa_{par}^{\sigma''}G\textbf{-Mod} \to \kappa\textbf{-Mod}  
        \]
        and 
        \[
            \hom_{\Lambda^e}(- \otimes_{B^{\sigma''}} \Lambda, M):\kappa_{par}^{\sigma''}G\textbf{-Mod} \to \kappa\textbf{-Mod}    
        \]
        are naturally isomorphic.
        \item For any left $\kappa_{par}^{\sigma''}G$-module $X$ the functors
        \[
            \hom_{\kappa_{par}^{\sigma''}G}(X, \hom_{A^e}(A,-)): \Lambda^e\textbf{-Mod} \to \kappa\textbf{-Mod}  
        \]
        and 
        \[
            \hom_{\Lambda^e}(X \otimes_{B^{\sigma''}} \Lambda, -): \Lambda^e\textbf{-Mod} \to \kappa\textbf{-Mod}    
        \]
        are naturally isomorphic.

    \end{enumerate}
\end{proposition}

\begin{corollary} \label{c_W2W1_W_iso}
    The functor $W_2 W_1$ is naturally $W$ isomorphic.
\end{corollary}

\begin{proof}
   If we set $X=B^\sigma$ in $(ii)$ of Proposition \ref{p_cohomology_functors_isomorphism} we obtain that
   \[
    W_2W_1 = \hom_{\kappa_{par}^{\sigma''}G}(B^\sigma, \hom_{A^e}(A,-)) \cong \hom_{\Lambda^e}(B^\sigma \otimes_{B^{\sigma''}} \Lambda, -) \cong W,
   \]
   where the later isomorphism is due to Lemma~\ref{l_AsmashHisoBAsmashH}.
\end{proof}

Since we want to use \cite[Theorem 10.47]{RotmanAnInToHoAl} we have to show that $W_1$ sends injective objects to $W_2$-acyclic $\kappa_{par}^{\sigma''}G$-modules. Let $Q$ be a injective $\Lambda^e$-module. Let $P_\bullet \to B^\sigma$ a projective resolution of $B^\sigma$ in $\kappa_{par}^{\sigma''}G$-\textbf{Mod}, then, for all $n \geq 1$,
\begin{align*}
    R^n W_2(W_1(Q)) 
    &= \operatorname{Ext}^n_{\kappa_{par}^{\sigma''}G}(B^\sigma, \hom_{A^e}(A,Q))\\
    &= H^n(\hom_{\kappa_{par}^{\sigma''}G}(P_\bullet, \hom_{A^e}(A,Q))) \\
   \text{by Proposition \ref{p_cohomology_functors_isomorphism} (i)} &=H^n(\hom_{\Lambda^e}(P_\bullet \otimes_{B^{\sigma''}} \Lambda,Q))=0.
\end{align*}
The last equality holds since $- \otimes_{B^{\sigma''}}\Lambda$ is exact (by Corollary \ref{l_Bflat}) and $\hom_{\Lambda^e}(-,Q)$ is exact ($Q$ is injective as $\Lambda^e$-module). Therefore, the following lemma has been proved

\begin{lemma} \label{l_InjectiveAcyclic}
    $W_1$ sends injectives in $W_2$-acyclics.
\end{lemma}

Let $X$ be a $\kappa_{par}^{\sigma''}G$-module. Analogously to Lemma \ref{B_Sigma_iso} we have that $\Omega_\sigma \otimes_{\kappa_{par}G} B \cong B^\sigma$ as left $\kappa_{par}^{\sigma''}G$-modules, thus  by \cite[Corollary 10.65]{RotmanAnInToHoAl} we obtain
\begin{align*}
    \operatorname{Ext}^n_{\kappa_{par}^{\sigma''}G}(B^\sigma, X)
    &= \operatorname{Ext}^n_{\kappa_{par}^{\sigma''}G}(\Omega_\sigma \otimes_{\kappa_{par}G} B, X) \\ 
    &= \operatorname{Ext}^n_{\kappa_{par}G}(B, \hom_{\kappa_{par}^{\sigma''}G}(\Omega_\sigma,X))
\end{align*}
In particular, if $X$ is a left $\Omega_\sigma$-module, then
\[
    \hom_{\kappa_{par}^{\sigma''}G}(\Omega_\sigma,X)=\hom_{\Omega_\sigma}(\Omega_\sigma,X)=X.
\]
Hence,
\begin{equation}
    \operatorname{Ext}^n_{\kappa_{par}^{\sigma''}G}(B^\sigma, X) = \operatorname{Ext}^n_{\kappa_{par}G}(B, X)= H^n_{par}(G,X).
\end{equation}

Analogously to Lemma \ref{l_AoMisOmegaModule} we have that $\hom_{A^e}(A,M)$ is an $\Omega_\sigma$-module. Therefore, the modules $R^n W_1(M)= H^n(A,M)$ are $\Omega_\sigma$-modules. Then, we obtain the following cohomological dual theorem of Theorem \ref{t_SSHH}

\begin{theorem} \label{t_SSCH}
    Let $A$ be a $\kappa$-algebra, $\theta$ a twisted partial action of a group $G$ on $A$, with twist $\sigma$. Then for any $A \ast_{\theta,\sigma} G$-bimodule $M$ there exists a third quadrant cohomological spectral sequence
    \[
        E_2^{p,q} = H^p_{par}(G,H^q(A,M) ) \Rightarrow H^{p+q}(A \ast_{\theta, \sigma} G, M). 
    \]
\end{theorem}

\begin{proof}
    Recalling Remark~\ref{r_Hochschild_cohomology} we know that 
    \[ 
       R^q W_1(-)=H^q(A,-),\,\, R^p G_2 (-) = H_{par}^p(G,-) \text{ and }R^{p+q}W(-)=H^{p+q}(A \ast_{\theta,\sigma} G, -). 
    \]
    Finally, by Corollary \ref{c_W2W1_W_iso}, Lemma \ref{l_InjectiveAcyclic} and \cite[Theorem 10.47]{RotmanAnInToHoAl} we obtain the desired spectral sequence.
\end{proof}

\section*{Acknowledgments}

The first named author was partially supported by 
Funda\c c\~ao de Amparo \`a Pesquisa do Estado de S\~ao Paulo (Fapesp), process n°:  2020/16594-0, and by  Conselho Nacional de Desenvolvimento Cient\'{\i}fico e Tecnol{\'o}gico (CNPq), process n°: 312683/2021-9. The second named author was supported by Fapesp, process n°: 2022/12963-7.

%%%%%%%%%%%%%%%%%%%%%%%%%%%%%%%%

\bibliographystyle{abbrv}
\bibliography{azu}

\end{document}